\documentclass[11pt]{amsart}
\usepackage{amsmath,amssymb,amsthm}
\usepackage[margin=1.05in]{geometry}
\usepackage{booktabs}
\usepackage{xcolor}
\usepackage{tikz}
\usepackage[colorlinks=true,linkcolor=blue!50!black,citecolor=green!40!black,urlcolor=blue!50!black]{hyperref}

\newtheorem{theorem}{Theorem}[section]
\newtheorem{proposition}[theorem]{Proposition}
\newtheorem{lemma}[theorem]{Lemma}
\newtheorem{corollary}[theorem]{Corollary}
\newtheorem{conjecture}[theorem]{Conjecture}
\theoremstyle{definition}
\newtheorem{definition}[theorem]{Definition}
\theoremstyle{remark}

\newtheorem{exa}{Example}[section]

\newcommand{\N}{\mathbb{N}}
\newcommand{\Z}{\mathbb{Z}}
\newcommand{\R}{\mathbb{R}}
\newcommand{\Imm}{\operatorname{Imm}}
\newcommand{\TNN}{\mathrm{TNN}}
\newcommand{\sgn}{\operatorname{sgn}}
\newcommand{\Ind}{\operatorname{Ind}}
\newcommand{\per}{\operatorname{per}}
\newcommand{\CF}{\mathrm{CF}}
\newcommand{\CKL}{C_{\mathrm{KL}}}
\newcommand{\CTNN}{C_{\mathrm{TNN}}}
\newcommand{\CJ}{C_{J}}
\newcommand{\Cp}[1]{C'_{#1}(1)}
\DeclareMathOperator{\ch}{ch}

\title[Hook monomial immanants are Kazhdan--Lusztig positive]{Hook monomial immanants\\ are Kazhdan--Lusztig positive}
\author{Keith McNulty}
\email{keith.mcnulty@gmail.com}
\date{15 November 2026}

\begin{document}

\begin{abstract}
Stembridge conjectured that every monomial immanant is non-negative on totally non-negative matrices, and Haiman proposed a stronger statement which at $q=1$ says that every monomial immanant is a non-negative integer combination of the Kazhdan--Lusztig immanants of Rhoades and Skandera. We prove this for every hook shape $(a,1^{n-a})$ and every $n$. More precisely, each of the per-vertex functions obtained by splitting Stembridge's formula for hook characters is Kazhdan--Lusztig positive.

The main tool is a positivity statement for products: a product of Kazhdan--Lusztig positive immanants of complementary submatrices is Kazhdan--Lusztig positive, which we deduce at $q=1$ from Braden's hyperbolic localization, following Billey and Braden. With Haiman's theorem for irreducible characters, it shows that every class function whose Frobenius characteristic lies in the semiring generated by Schur functions and hook monomial symmetric functions is Kazhdan--Lusztig positive; for example, Skandera's chain of normalised hook immanants is Kazhdan--Lusztig positive for every $n$.

We place these results in a comparison of three cones of class functions, and show that the remaining obstructions to Haiman's conjecture lie at singular Schubert varieties that are not reached by an iterated doubling identity of Abreu and Nigro, for which we give an explicit criterion. Finally, on the full span of permutation monomials, an element that is non-negative on totally non-negative matrices can have a negative coefficient at the Kazhdan--Lusztig immanant $\Imm_w$ if and only if the Schubert variety $X_w$ is singular; Skandera proved the smooth half.
\end{abstract}

\maketitle

\section{Introduction}\label{sec:intro}

A real matrix is \emph{totally non-negative} (TNN) if all its minors are non-negative; $\TNN_n$ denotes the set of such $n\times n$ matrices. For a function $f\colon S_n\to\R$ and an $n\times n$ matrix $x$ put
\[
  d_f(x)=\sum_{\sigma\in S_n} f(\sigma)\,x_{1,\sigma(1)}\cdots x_{n,\sigma(n)} .
\]
For an irreducible character $\chi^\lambda$ this is the immanant $\Imm_{\chi^\lambda}$, and the permanent and the determinant are the cases $\lambda=(n)$ and $\lambda=(1^n)$. Stembridge \cite{Ste91} proved that every irreducible-character immanant is non-negative on $\TNN_n$. The \emph{monomial immanants} $\Imm_\lambda=d_{\phi^\lambda}$, where $\phi^\lambda$ is the monomial virtual character (Section~\ref{sec:prelim}), refine them: each $\chi^\mu$ is a non-negative integer combination of the $\phi^\lambda$.

\begin{conjecture}[Stembridge {\cite[Conj.~2.1]{Ste92}}]\label{conj:MN}
For every $\lambda\vdash n$ and every $A\in\TNN_n$, $\Imm_\lambda(A)\ge 0$.
\end{conjecture}

The conjecture is known for hooks (Dong--Wu--Gao \cite{DWG}), for shapes with at most two columns (Clearman--Shelton--Skandera \cite{CSS}), and for rectangles and $(2,1^{n-2})$ (Stembridge \cite{Ste92}).

Haiman \cite{Hai} proposed a stronger statement in the Hecke algebra. Let $P_{v,w}(q)$ be the Kazhdan--Lusztig polynomials of $S_n$, and let $\Cp w=\sum_{v\le w}P_{v,w}(1)\,v$ be the Kazhdan--Lusztig basis of $\Z[S_n]$ at $q=1$. Rhoades and Skandera \cite{RS06} introduced Kazhdan--Lusztig immanants $\Imm_w$, $w\in S_n$, which are non-negative on $\TNN_n$ and satisfy $d_f=\sum_wf(\Cp w)\Imm_w$ for every $f$. We call $d_f$ (or $f$) \emph{KL-positive} if $f(\Cp w)\ge0$ for every $w$; then $d_f\ge0$ on $\TNN_n$.

\begin{conjecture}[Haiman, Stembridge; $q=1$]\label{conj:H}
For every $\lambda\vdash n$ and $w\in S_n$, $\phi^\lambda(\Cp w)\ge 0$. Equivalently, every monomial immanant is KL-positive.
\end{conjecture}

Haiman's conjecture \cite[Conj.~2.1]{Hai} is stronger: for the monomial characters $\phi^\lambda_q$ of the Hecke algebra, $\phi^\lambda_q(q^{\ell(w)/2}C'_w)$ has non-negative integer coefficients and is symmetric and unimodal, and Conjecture~\ref{conj:H} is its value at $q=1$. Haiman verified it by computer for $n\le7$ \cite[\S4]{Hai}. Conjecture~\ref{conj:H} implies Conjecture~\ref{conj:MN}. It holds for shapes with at most two columns, by the Temperley--Lieb expansions of Clearman, Shelton and Skandera \cite{CSS} together with \cite[Prop.~3.2]{RS06}. Our main theorem adds every hook.

\begin{theorem}\label{thm:main}
For every $n$ and $1\le a\le n$, the monomial immanant $\Imm_{(a,1^{n-a})}$ is KL-positive. More precisely, for every $i\in[n]$ the per-vertex function
\[
f^{(a)}_i(\sigma)=(-1)^{a-1}\sgn(\sigma)\,[\text{the cycle of $\sigma$ through $i$ has length}\ge a]
\]
is KL-positive, and $\phi^{(a,1^{n-a})}=\sum_if^{(a)}_i$ for $a\ge2$.
\end{theorem}

The decomposition into the $f^{(a)}_i$ follows from Stembridge's formula \cite[Prop.~1.1]{Ste92}, which for hooks counts the vertices in cycles of length $\ge a$ (see \cite[Lemma~2.1]{DWG}); the $f^{(a)}_i$ are the single-permutation versions of the marked cycle sums of \cite[Def.~2.2]{DWG}. Dong, Wu and Gao \cite[proof of Thm.~1.2, \S2.5]{DWG} proved by Schur complements that each $d_{f^{(a)}_i}$ is non-negative on every matrix whose minors of order at most $\ell(\lambda)=n-a+1$ are non-negative, answering Stembridge's question \cite[Question~2.9]{Ste92} for hooks; in particular $d_{f^{(a)}_i}\ge0$ on $\TNN_n$. On $\TNN_n$, Theorem~\ref{thm:main} explains this positivity: each $d_{f^{(a)}_i}$ is a non-negative integer combination of Kazhdan--Lusztig immanants. It does not give their result under the weaker hypothesis, since a Kazhdan--Lusztig immanant need not be non-negative there: for $n=4$ and $\lambda=(3,1)$, $\Imm_{1342}$ and $\Imm_{1423}$ occur in the expansion of $\Imm_\lambda$ with coefficient $3$, and the matrix $A$ with rows $(3,1,1,0)$, $(3,1,1,1)$, $(2,1,1,1)$, $(2,1,1,2)$ has all $1\times1$ and $2\times2$ minors non-negative but $\Imm_{1342}(A)=\Imm_{1423}(A^T)=-1$. For $a=2$ the expansion is explicit, $\Imm_{(2,1^{n-2})}=2\sum_k\Imm_{s_k}+\sum_k(\Imm_{s_ks_{k+1}}+\Imm_{s_{k+1}s_k})$ (Corollary~\ref{cor:a2}). For $a\ge3$ the coefficients are not explicit, and the proof is an induction over per-vertex functions composed with rotations of consecutive values (Section~\ref{sec:hooks}). Its main tool is the following theorem, which we prove at $q=1$ from Braden's hyperbolic localization \cite{Bra03}, following Billey and Braden \cite{BB03}.

\begin{theorem}\label{thm:prod}
Let $I,T\subseteq[n]$ with $|I|=|T|$. If $F$ is KL-positive on the submatrix $x[I,T]$ and $G$ is KL-positive on the complementary submatrix $x[I^c,T^c]$, then $FG$ is KL-positive. In particular, the induction product of two KL-positive class functions is KL-positive.
\end{theorem}

We prove Theorem~\ref{thm:prod} as Theorem~\ref{thm:R2}. By Haiman \cite[Lemma~1.1]{Hai}, every irreducible character is KL-positive. Theorems~\ref{thm:main} and~\ref{thm:prod} therefore give a semiring of KL-positive class functions.

\begin{corollary}\label{cor:semiring}
Let $\mathcal R$ be the semiring generated by the Schur functions and the hook monomial symmetric functions $m_{(a,1^b)}$. Every class function whose Frobenius characteristic lies in $\mathcal R$ is KL-positive, so its immanant is non-negative on $\TNN_n$.
\end{corollary}

For example, Skandera \cite[Thm.~27]{Ska25} proved that the normalised hook immanants form a chain $\per\ge N_{(n-1,1)}\ge\cdots\ge N_{(1^n)}=\det$ on $\TNN_n$. We show that each difference has Frobenius characteristic $\sum_jh_jm_{(k-j,1^{n-k})}$ up to a positive factor, so the chain holds at the level of Kazhdan--Lusztig immanants (Corollary~\ref{cor:skandera}). Since the power sums $p_k=m_{(k)}$ are hook monomials, the Murnaghan--Nakayama rule turns every product $s_\mu p_k$ into a signed inequality between irreducible-character immanants (Corollary~\ref{cor:MN}); for staircases $\mu$ this is a two-term inequality, such as $N_{(5,2,1)}\ge N_{(3,2,1,1,1)}$ on $\TNN_8$.

\subsection*{Three cones}
These results are best seen inside a comparison of three cones in the space $\CF_n$ of real class functions on $S_n$:
\begin{align*}
\CKL&=\{\phi:\ \phi(\Cp w)\ge0\text{ for all }w\in S_n\},\\
\CTNN&=\{\phi:\ d_\phi\ge0\text{ on }\TNN_n\},\\
\CJ&=\{\phi:\ d_\phi(J_\mu)\ge0\text{ for all }\mu\vdash n\},
\end{align*}
where $J_\mu$ is the block-diagonal matrix with all-ones blocks of sizes $\mu_1,\mu_2,\dots$. For every $n$, $\CKL\subseteq\CTNN\subseteq\CJ$, and $\CJ$ is the simplicial cone spanned by the $\phi^\lambda$ (Section~\ref{sec:cones}). So Conjecture~\ref{conj:MN} is the equality $\CTNN=\CJ$, which says that the matrices $J_\mu$ test total non-negativity of class-function immanants, and Conjecture~\ref{conj:H} is the equality $\CKL=\CJ$. Since the cones are nested, Conjecture~\ref{conj:H} is equivalent to $\CKL=\CTNN=\CJ$: all three cones coincide.

A companion preprint \cite{McN} extends Haiman's verification to $n\le10$ by an exact computation, which verifies his conjecture in full for these $n$. In particular Conjecture~\ref{conj:H} holds for $n\le10$, and all three cones coincide. The present paper describes $\CKL$ for all $n$ from three directions. Corollary~\ref{cor:semiring} describes elements of $\CKL$. Known results show that the coefficients $\phi^\lambda(\Cp w)$ can only fail to be non-negative at singular Schubert varieties. Clearman, Hyatt, Shelton and Skandera \cite{CHSS}, and Abreu and Nigro \cite[Thm.~1.6]{AN}, reduce a smooth ($3412$- and $4231$-avoiding) $w$ to a codominant one. There the coefficients are values of chromatic symmetric functions, which are non-negative by Hikita's proof of the Stanley--Stembridge conjecture \cite{Hik}, as Skandera observed \cite[Prop.~25]{Ska25} (Proposition~\ref{thm:smooth}). The next theorem reduces the problem further.

\begin{theorem}\label{thm:smoothclean}
If $y$ is joined to a smooth permutation by $k$ \emph{clean steps} (Section~\ref{sec:clean}), then $\phi^\lambda(\Cp y)=2^k\phi^\lambda(\Cp{w'})\ge0$ for a codominant $w'$ and every $\lambda$. So Conjecture~\ref{conj:H} holds for $n$ iff $\phi^\lambda(\Cp w)\ge0$ for every singular $w$ that is not the top of a clean step. If $\deg P_{e,w}\le1$, a step is clean iff $w$ has no \emph{bad cover} there, an explicit local condition.
\end{theorem}

Theorem~\ref{thm:smoothclean} iterates a doubling identity of Abreu and Nigro \cite{AN}. Finally, the restriction to class functions matters. For class functions $\CKL=\CTNN$ when $n\le10$, and for every $n$ if Conjecture~\ref{conj:H} holds. On the span $\Pi_n$ of all permutation monomials, the analogous cones differ already for $n=4$ \cite[\S7]{RS06}: for example, Skandera and Zelevinsky observed that $\Imm_{3214}+\Imm_{1432}-\Imm_{3412}$ is non-negative on $\TNN_4$ but not KL-positive (reported in \cite[(5.1)]{Ska08}), and Theorem~\ref{thm:beyond} gives further examples for $n=5$, such as $\Imm_{21543}+\Imm_{54321}-\Imm_{52431}$. Smoothness appears here too: by a theorem of Skandera \cite[Thm.~5.5]{Ska08}, a negative coefficient at a smooth $w$ never occurs in a TNN element of $\Pi_n$. We prove the converse.

\begin{theorem}\label{thm:isolatedintro}
For every $n$ and every singular $w\in S_n$ there is an element of $\Pi_n$ that is non-negative on $\TNN_n$ and has a negative coefficient at $\Imm_w$. Together with Skandera's theorem, this characterises the smooth $w$ by total non-negativity (Theorem~\ref{thm:isolated}).
\end{theorem}

The proof lifts two elements of $\Pi_4$, given by the inequalities $\Imm_{3412}^2\le\Imm_{1432}\Imm_{3214}$ and $\Imm_{4231}^2\le3\Imm_{2143}\Imm_{4321}$ on $\TNN_4$, along an occurrence of $3412$ or $4231$. Theorem~\ref{thm:prod} controls the lift, and the tangent spaces of Lakshmibai and Seshadri handle the $4231$-avoiding case.

The paper is organised linearly. Section~\ref{sec:prelim} recalls Kazhdan--Lusztig immanants and the Temperley--Lieb colouring rule, Section~\ref{sec:prod} proves Theorem~\ref{thm:prod}, and Section~\ref{sec:hooks} proves Theorem~\ref{thm:main}. Section~\ref{sec:consequences} derives Corollary~\ref{cor:semiring} and the hook chain, and Section~\ref{sec:cones} sets up the three cones. Sections~\ref{sec:smoothsing} and~\ref{sec:twocol} treat smooth and singular permutations and two further families of shapes, Section~\ref{sec:beyond} the full span $\Pi_n$ and Theorem~\ref{thm:isolatedintro}, and Section~\ref{sec:open} collects conjectures and open problems.

\emph{Computational materials.} The proof of Theorem~\ref{thm:beyond} uses an exact computation in explicit charts. The other computations reported in this paper are cross-checks of the proofs or evidence for the conjectures: checks of Theorem~\ref{thm:prod} and of the lemmas of Section~\ref{sec:hooks}, of Corollary~\ref{cor:skandera}, of the clean-step criterion of Section~\ref{sec:clean}, of Theorem~\ref{thm:isolated} and Lemma~\ref{lem:smoothswap}, and the computations behind Conjectures~\ref{conj:B} and~\ref{conj:rooted}. The code for all of them, together with a script that checks every number in the worked examples and figures, is archived with the code of the companion paper \cite{McN} as \cite{McNcode}; the companion paper describes the computational framework and its verification.

\section{Preliminaries}\label{sec:prelim}

\subsection{Class functions}
For a composition $\alpha$ of $n$ let $S_\alpha=S_{\alpha_1}\times S_{\alpha_2}\times\cdots$ be the Young subgroup preserving the consecutive blocks of sizes $\alpha_1,\alpha_2,\dots$, and let $\alpha!=|S_\alpha|$. For $\mu\vdash n$ put $\eta^\mu=\Ind_{S_\mu}^{S_n}1$. With $\langle\phi,\psi\rangle=\frac1{n!}\sum_\sigma\phi(\sigma)\psi(\sigma)$, the \emph{monomial characters} $\phi^\lambda$ are the dual basis: $\langle\phi^\lambda,\eta^\mu\rangle=\delta_{\lambda\mu}$. Under the Frobenius characteristic $\ch$, $\eta^\mu\mapsto h_\mu$, $\chi^\lambda\mapsto s_\lambda$ and $\phi^\lambda\mapsto m_\lambda$. By Young's rule $\langle\chi^a,\eta^\mu\rangle=K_{a,\mu}$, the Kostka number. The induction product $\phi\cdot\psi=\Ind_{S_k\times S_{n-k}}^{S_n}(\phi\times\psi)$ satisfies $\ch(\phi\cdot\psi)=\ch\phi\,\ch\psi$, and
\begin{equation}\label{eq:indprod}
d_{\phi\cdot\psi}(x)=\sum_{|T|=k}d_\phi(x[T,T])\,d_\psi(x[T^c,T^c]).
\end{equation}

\subsection{Kazhdan--Lusztig immanants}
Rhoades and Skandera \cite{RS06} define
\[
  \Imm_w(x)=\sum_{v\ge w}(-1)^{\ell(v)-\ell(w)}P_{w_0v,w_0w}(1)\,x_{1,v(1)}\cdots x_{n,v(n)}.
\]
They prove that $\Imm_w\ge0$ on $\TNN_n$ and that, for every $f\colon S_n\to\R$, extended linearly to $\R[S_n]$,
\begin{equation}\label{eq:duality}
  d_f=\sum_{w\in S_n}f(\Cp w)\,\Imm_w .
\end{equation}
So the $\Imm_w$ are the basis of the span $\Pi_n$ of permutation monomials dual to $\{\Cp w\}$. Comparing coefficients of the monomials shows that \eqref{eq:duality} is equivalent to the Kazhdan--Lusztig inversion formula $\sum_{u\le w\le v}(-1)^{\ell(v)-\ell(w)}P_{u,w}(1)P_{w_0v,w_0w}(1)=\delta_{u,v}$ \cite[Thm.~3.1]{KL79}. We write $c_{\lambda,w}=\phi^\lambda(\Cp w)$, so that $\Imm_\lambda=\sum_wc_{\lambda,w}\Imm_w$. For a square submatrix $x[R,C]$ we index its Kazhdan--Lusztig immanants by $S_{|R|}$, through the order-preserving bijections of $R$ and $C$ with $[|R|]$, and call a polynomial \emph{KL-positive on $x[R,C]$} if it is a non-negative real combination of these. For integer-valued $f$ the coefficients $f(\Cp w)$ are integers, so for characters, and for the per-vertex functions of Theorem~\ref{thm:main}, KL-positivity means non-negative integer coefficients.

Two special cases are used repeatedly. Since $P_{u,w_0}=1$ for every $u$, $\Imm_e=\det$, and \eqref{eq:duality} with $f=\sgn$ gives $\sgn(\Cp w)=\delta_{w,e}$. Since $P_{v,w}\ge0$ coefficientwise, every monomial $x^v=x_{1,v(1)}\cdots x_{n,v(n)}=d_{\delta_v}$ is KL-positive, with $\delta_v(\Cp w)=P_{v,w}(1)$.

\subsection{The Temperley--Lieb colouring rule}
Let $\tau$ be a Temperley--Lieb diagram on $n+n$ vertices: left vertices $1,\dots,n$ and right vertices $1',\dots,n'$, joined in pairs by non-crossing edges. Edges between two left vertices are \emph{cups}, edges between two right vertices are \emph{caps}, and the others are \emph{through-lines}. Rhoades and Skandera \cite{RS05} attach to $\tau$ a Temperley--Lieb immanant $\Imm_\tau$, and \cite[Prop.~3.2]{RS06} shows that $\Imm_\tau=\Imm_w$ for the $321$-avoiding $w$ with the same reduced word as $\tau$. This is a bijection between diagrams and $321$-avoiding permutations \cite{FG}. Their colouring rule \cite[Prop.~4.3]{RS05} expands a product of complementary minors: for $|I|=|I'|$,
\[
\det x[I,I']\,\det x[I^c,I'^c]=\sum_\tau\Imm_\tau(x),
\]
the sum over the diagrams $\tau$ in which every edge joins a black vertex to a white one, where the left vertex $i$ is black iff $i\in I$ and the right vertex $j'$ is black iff $j\notin I'$. We need the case $|I|=1$ (Figure~\ref{fig:colour}).

\begin{proposition}\label{prop:colour}
Let $v,c\in[n]$. Then $x_{vc}\det x[\hat v,\hat c]=\sum_\tau\Imm_\tau$, where $\tau$ runs over
\begin{itemize}
\item the identity diagram, if $c=v$;
\item the diagrams $\tau_{\epsilon,\delta}$ with one cup $\{v,v+\epsilon\}$, one cap $\{c',(c+\delta)'\}$ and order-preserving through-lines, for $\epsilon,\delta\in\{\pm1\}$ with $v+\epsilon,c+\delta\in[n]$.
\end{itemize}
In particular the non-identity diagrams are determined by their cup and cap, and the diagram with cap $\{k',(k+1)'\}$ occurs for $c=k$ and $c=k+1$ and for no other $c$.
\end{proposition}

\begin{proof}
Here $I=\{v\}$ and $I'=\{c\}$, so the black vertices are $v$ on the left and every $j'\ne c'$ on the right. Every cup contains the only black left vertex $v$, so there is at most one cup, and there are as many caps as cups. If there is no cup, $\tau$ is the identity diagram, and its edge from $v$ must end at the white vertex $c'$, so $c=v$. If there is one cup, it contains $v$, and the cap contains $c'$, the only white right vertex. A cup or cap that encloses another vertex would cross the through-line from that vertex, so the cup is $\{v,v\pm1\}$ and the cap is $\{c',(c\pm1)'\}$. The remaining $n-2$ left vertices are joined to the remaining $n-2$ right vertices by non-crossing through-lines, which forces the order-preserving matching. Each of these joins a white left vertex to a black right vertex, so all such diagrams are allowed.
\end{proof}

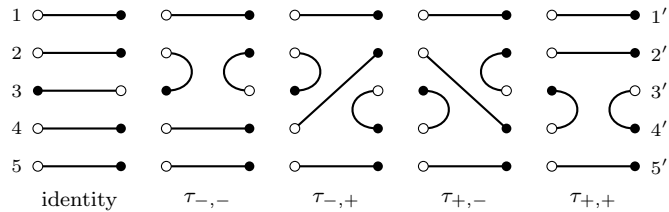
\begin{figure}[ht]
\centering
\begin{tikzpicture}[scale=0.5,
  blk/.style={circle,fill=black,inner sep=1.3pt},
  wht/.style={circle,draw=black,fill=white,inner sep=1.3pt}]
  \begin{scope}[xshift=0.0cm]
    \draw[thick] (0,-1)--(2.2,-1);
    \draw[thick] (0,-2)--(2.2,-2);
    \draw[thick] (0,-3)--(2.2,-3);
    \draw[thick] (0,-4)--(2.2,-4);
    \draw[thick] (0,-5)--(2.2,-5);
    \node[wht] at (0,-1) {};
    \node[blk] at (2.2,-1) {};
    \node[wht] at (0,-2) {};
    \node[blk] at (2.2,-2) {};
    \node[blk] at (0,-3) {};
    \node[wht] at (2.2,-3) {};
    \node[wht] at (0,-4) {};
    \node[blk] at (2.2,-4) {};
    \node[wht] at (0,-5) {};
    \node[blk] at (2.2,-5) {};
    \node[font=\scriptsize] at (1.1,-5.9) {identity};
  \end{scope}
  \begin{scope}[xshift=3.4cm]
    \draw[thick] (0,-2) .. controls (0.9,-2) and (0.9,-3) .. (0,-3);
    \draw[thick] (2.2,-2) .. controls (1.3,-2) and (1.3,-3) .. (2.2,-3);
    \draw[thick] (0,-1)--(2.2,-1);
    \draw[thick] (0,-4)--(2.2,-4);
    \draw[thick] (0,-5)--(2.2,-5);
    \node[wht] at (0,-1) {};
    \node[blk] at (2.2,-1) {};
    \node[wht] at (0,-2) {};
    \node[blk] at (2.2,-2) {};
    \node[blk] at (0,-3) {};
    \node[wht] at (2.2,-3) {};
    \node[wht] at (0,-4) {};
    \node[blk] at (2.2,-4) {};
    \node[wht] at (0,-5) {};
    \node[blk] at (2.2,-5) {};
    \node[font=\scriptsize] at (1.1,-5.9) {$\tau_{-,-}$};
  \end{scope}
  \begin{scope}[xshift=6.8cm]
    \draw[thick] (0,-2) .. controls (0.9,-2) and (0.9,-3) .. (0,-3);
    \draw[thick] (2.2,-3) .. controls (1.3,-3) and (1.3,-4) .. (2.2,-4);
    \draw[thick] (0,-1)--(2.2,-1);
    \draw[thick] (0,-4)--(2.2,-2);
    \draw[thick] (0,-5)--(2.2,-5);
    \node[wht] at (0,-1) {};
    \node[blk] at (2.2,-1) {};
    \node[wht] at (0,-2) {};
    \node[blk] at (2.2,-2) {};
    \node[blk] at (0,-3) {};
    \node[wht] at (2.2,-3) {};
    \node[wht] at (0,-4) {};
    \node[blk] at (2.2,-4) {};
    \node[wht] at (0,-5) {};
    \node[blk] at (2.2,-5) {};
    \node[font=\scriptsize] at (1.1,-5.9) {$\tau_{-,+}$};
  \end{scope}
  \begin{scope}[xshift=10.2cm]
    \draw[thick] (0,-3) .. controls (0.9,-3) and (0.9,-4) .. (0,-4);
    \draw[thick] (2.2,-2) .. controls (1.3,-2) and (1.3,-3) .. (2.2,-3);
    \draw[thick] (0,-1)--(2.2,-1);
    \draw[thick] (0,-2)--(2.2,-4);
    \draw[thick] (0,-5)--(2.2,-5);
    \node[wht] at (0,-1) {};
    \node[blk] at (2.2,-1) {};
    \node[wht] at (0,-2) {};
    \node[blk] at (2.2,-2) {};
    \node[blk] at (0,-3) {};
    \node[wht] at (2.2,-3) {};
    \node[wht] at (0,-4) {};
    \node[blk] at (2.2,-4) {};
    \node[wht] at (0,-5) {};
    \node[blk] at (2.2,-5) {};
    \node[font=\scriptsize] at (1.1,-5.9) {$\tau_{+,-}$};
  \end{scope}
  \begin{scope}[xshift=13.6cm]
    \draw[thick] (0,-3) .. controls (0.9,-3) and (0.9,-4) .. (0,-4);
    \draw[thick] (2.2,-3) .. controls (1.3,-3) and (1.3,-4) .. (2.2,-4);
    \draw[thick] (0,-1)--(2.2,-1);
    \draw[thick] (0,-2)--(2.2,-2);
    \draw[thick] (0,-5)--(2.2,-5);
    \node[wht] at (0,-1) {};
    \node[blk] at (2.2,-1) {};
    \node[wht] at (0,-2) {};
    \node[blk] at (2.2,-2) {};
    \node[blk] at (0,-3) {};
    \node[wht] at (2.2,-3) {};
    \node[wht] at (0,-4) {};
    \node[blk] at (2.2,-4) {};
    \node[wht] at (0,-5) {};
    \node[blk] at (2.2,-5) {};
    \node[font=\scriptsize] at (1.1,-5.9) {$\tau_{+,+}$};
  \end{scope}
  \node[left,font=\tiny] at (-0.15,-1) {$1$};
  \node[right,font=\tiny] at (15.95,-1) {$1'$};
  \node[left,font=\tiny] at (-0.15,-2) {$2$};
  \node[right,font=\tiny] at (15.95,-2) {$2'$};
  \node[left,font=\tiny] at (-0.15,-3) {$3$};
  \node[right,font=\tiny] at (15.95,-3) {$3'$};
  \node[left,font=\tiny] at (-0.15,-4) {$4$};
  \node[right,font=\tiny] at (15.95,-4) {$4'$};
  \node[left,font=\tiny] at (-0.15,-5) {$5$};
  \node[right,font=\tiny] at (15.95,-5) {$5'$};
\end{tikzpicture}
\caption{Proposition~\ref{prop:colour} for $n=5$, $v=c=3$: the five diagrams in $x_{33}\det x[\hat3,\hat3]$. Black vertices are filled: the left vertex $3$ and every right vertex except $3'$. Each edge joins a black vertex to a white one. The identity diagram appears because $c=v$, and the other four have cup $\{3,3\pm1\}$ and cap $\{3',(3\pm1)'\}$, with order-preserving through-lines. They are the Temperley--Lieb immanants of $s_2$, $s_3$, $s_2s_3$ and $s_3s_2$ (Lemma~\ref{lem:a2}).}\label{fig:colour}
\end{figure}

\section{Products of immanants of complementary submatrices}\label{sec:prod}

\begin{theorem}\label{thm:R2}
Let $I,T\subseteq[n]$ with $|I|=|T|=k$. If $F$ is KL-positive on $x[I,T]$ and $G$ is KL-positive on $x[I^c,T^c]$, then $FG$ is KL-positive. In particular:
\begin{enumerate}
\item if $G$ is KL-positive on $x[\hat\imath,\hat\jmath]$, then $x_{ij}G$ is KL-positive;
\item for every bijection $\gamma$ of a set $T\subseteq[n]$, $x_\gamma\det x[T^c,T^c]$ is KL-positive, where $x_\gamma=\prod_{t\in T}x_{t,\gamma(t)}$;
\item if $\phi\in\CF_k$ and $\psi\in\CF_{n-k}$ are KL-positive, then so is $\phi\cdot\psi$.
\end{enumerate}
\end{theorem}

\begin{proof}
By linearity we may take $F=\Imm_u(x[I,T])$ and $G=\Imm_v(x[I^c,T^c])$ with $u\in S_k$, $v\in S_{n-k}$. For $\sigma\in S_n$ with $\sigma(I)=T$ let $\varphi(\sigma)=(\sigma_I,\sigma_{I^c})\in S_k\times S_{n-k}$, where $\sigma_I$ is the restriction $I\to T$ and $\sigma_{I^c}$ the restriction $I^c\to T^c$, both standardised. Then $FG=d_f$ with $f(\sigma)=[\sigma(I)=T]\,f_u(\sigma_I)f_v(\sigma_{I^c})$, where $\Imm_u=d_{f_u}$ and $f_u(\Cp{u'})=\delta_{uu'}$ on $S_k$, and similarly for $v$. By \eqref{eq:duality} the coefficient of $\Imm_w$ in $FG$ is
\[
f(\Cp w)=\sum_{\sigma(I)=T}P_{\sigma,w}(1)\,f_u(\sigma_I)f_v(\sigma_{I^c}).
\]
Let $W'\subseteq S_n$ be the stabiliser of the set of values $T$. It is generated by the reflections whose roots are orthogonal to the indicator vector $\mathbf 1_T$, so it is a parabolic reflection subgroup in the sense of Billey and Braden, and $W'\cong S_k\times S_{n-k}$. Its cosets $W'z$ are the sets $\{\sigma:\sigma(I)=T\}$, $I=z^{-1}(T)$, and the pattern map restricts to the bijection $\varphi$ of $W'z$ onto $W'$. Braden's theorem on hyperbolic localization \cite{Bra03}, applied to the fixed component of a one-parameter subgroup indexed by $W'z$, gives (see the proof of \cite[Thm.~4]{BB03}, equation (2))
\begin{equation}\label{eq:BB}
P_{\sigma,w}(1)=\sum_{y\in W'z}a_y\,P_{\varphi(\sigma),\varphi(y)}(1)\qquad(\sigma\in W'z)
\end{equation}
for integers $a_y=a_y(w)\ge0$, the multiplicities of the simple summands of the hyperbolic localization of the intersection cohomology sheaf of the Schubert variety $X_w$. They depend on $w$ and the coset, but not on $\sigma$, so the same $a_y$ serve for every $\sigma\in W'z$. The shifts in this decomposition are even, so passing to Euler characteristics introduces no signs. The Kazhdan--Lusztig polynomials of $W'$ are products, so $\sum_{\sigma(I)=T}P_{\sigma,w}(1)\,\varphi(\sigma)=\sum_ya_y\,\Cp{\varphi(y)_1}\otimes\Cp{\varphi(y)_2}$ in $\Z[S_k]\otimes\Z[S_{n-k}]$. Applying $f_u\otimes f_v$ shows that the coefficient of $\Imm_w$ is $a_y\ge0$ for the $y$ with $\varphi(y)=(u,v)$.

Part~(1) is the case $I=\{i\}$, $T=\{j\}$, $F=x_{ij}=\Imm_e(x[\{i\},\{j\}])$. Part~(2) follows from~(1) by adding one entry at a time, starting from $\det x[T^c,T^c]=\Imm_e$. Part~(3) follows from \eqref{eq:indprod}: each term $d_\phi(x[T,T])\,d_\psi(x[T^c,T^c])$ is KL-positive by the main statement with $I=T$.
\end{proof}

\begin{exa}\label{exa:R2}
Take $n=3$, $I=\{1\}$, $T=\{2\}$, $F=x_{12}$ and $G=\det x[\hat1,\hat2]$. The coset $\{\sigma:\sigma(1)=2\}$ is $\{213,231\}$, and $\varphi(213)=e$, $\varphi(231)=s_1$ in $S_2$. Write $a=a_{213}(w)$ and $b=a_{231}(w)$. Since $P_{e,e}=P_{e,s_1}=P_{s_1,s_1}=1$ and $P_{s_1,e}=0$, the identity \eqref{eq:BB} reads $P_{213,w}(1)=a+b$ and $P_{231,w}(1)=b$. Every Kazhdan--Lusztig polynomial of $S_3$ is $1$ on its Bruhat interval, so:
\begin{center}
\begin{tabular}{lcccccc}
\toprule
$w$ & $123$ & $132$ & $213$ & $231$ & $312$ & $321$\\
\midrule
$(P_{213,w},P_{231,w})$ & $(0,0)$ & $(0,0)$ & $(1,0)$ & $(1,1)$ & $(1,0)$ & $(1,1)$\\
$(a,b)$ & $(0,0)$ & $(0,0)$ & $(1,0)$ & $(0,1)$ & $(1,0)$ & $(0,1)$\\
\bottomrule
\end{tabular}
\end{center}
The coefficient of $\Imm_w$ in $FG$ is $a_y$ for the $y$ with $\varphi(y)=(e,e)$, that is $a=a_{213}(w)$. So $x_{12}\det x[\hat1,\hat2]=\Imm_{213}+\Imm_{312}$, which is also what Proposition~\ref{prop:colour} gives with $v=1$, $c=2$: the cup $\{1,2\}$ with the caps $\{1',2'\}$ and $\{2',3'\}$.
\end{exa}

For $k=1$, the coset is $\{\sigma:\sigma(i)=j\}$ and $\varphi(\sigma)$ is $\sigma$ with position $i$ and value $j$ deleted, standardised. This is the form used in Section~\ref{sec:hooks}. For products of two complementary minors, Proposition~\ref{prop:colour} and its general form \cite[Prop.~4.3]{RS05} give the expansion explicitly. Theorem~\ref{thm:R2} proves at $q=1$ the positivity of products of Kazhdan--Lusztig immanants of complementary submatrices, which Rhoades and Skandera \cite[(3.4)]{RSbit} state, citing Lusztig, as a property of the dual canonical basis.

\section{Hooks}\label{sec:hooks}

\subsection{Per-vertex functions}
For $\sigma\in S_n$ and $i\in[n]$ let $\ell_i(\sigma)$ be the length of the cycle of $\sigma$ through $i$, and put
\[
f^{(a)}_i(\sigma)=(-1)^{a-1}\sgn(\sigma)\,[\ell_i(\sigma)\ge a].
\]
Stembridge's formula \cite[Prop.~1.1]{Ste92} (see \cite[Lemma~2.1]{DWG}) gives, with $\kappa_{a,k}=1$ for $a>1$ and $\kappa_{1,k}=k+1$,
\begin{equation}\label{eq:stembridge}
\kappa_{a,k}\,\phi^{(a,1^k)}=\sum_{i=1}^n f^{(a)}_i .
\end{equation}
Fix $i$. For $L\ge1$ let $G_L$ be the sum, over $L$-subsets $T\ni i$ and $L$-cycles $\gamma$ on $T$, of $x_\gamma\det x[T^c,T^c]$. Sorting permutations by the cycle through $i$ gives
\begin{equation}\label{eq:cycle}
\det x=G_1-G_2+G_3-\cdots,\qquad d_{f^{(a)}_i}=\sum_{L\ge a}(-1)^{L-a}G_L .
\end{equation}
Each $G_L$ is KL-positive by Theorem~\ref{thm:R2}(2), but $d_{f^{(a)}_i}$ is an alternating sum of them. For instance,
\begin{equation}\label{eq:f2}
d_{f^{(2)}_i}=x_{ii}\det x[\hat\imath,\hat\imath]-\det x .
\end{equation}

\subsection{The case \texorpdfstring{$a=2$}{a=2}}
\begin{lemma}\label{lem:a2}
$f^{(2)}_i(\Cp w)=1$ if $w\in\{s_{i-1},s_i,s_{i-1}s_i,s_is_{i-1}\}$ (those that exist), and $f^{(2)}_i(\Cp w)=0$ otherwise.
\end{lemma}

\begin{proof}
By Proposition~\ref{prop:colour} with $v=c=i$, $x_{ii}\det x[\hat\imath,\hat\imath]$ is the sum of $\Imm_\tau$ over the identity diagram and the diagrams with cup $\{i,i\pm1\}$ and cap $\{i',(i\pm1)'\}$. These four are the Temperley--Lieb generators $t_{i-1}$, $t_i$ and the products $t_{i-1}t_i$, $t_it_{i-1}$, and their $321$-avoiding permutations are $s_{i-1}$, $s_i$, $s_{i-1}s_i$, $s_is_{i-1}$. The identity diagram gives $\Imm_e=\det$, which cancels in \eqref{eq:f2}.
\end{proof}

\begin{corollary}\label{cor:a2}
For every $n$, $\Imm_{(2,1^{n-2})}=2\sum_{k=1}^{n-1}\Imm_{s_k}+\sum_{k=1}^{n-2}\bigl(\Imm_{s_ks_{k+1}}+\Imm_{s_{k+1}s_k}\bigr)$.
\end{corollary}

\begin{proof}
Sum Lemma~\ref{lem:a2} over $i$, using \eqref{eq:stembridge}: $s_k$ occurs for $i=k$ and $i=k+1$, and $s_ks_{k+1}$, $s_{k+1}s_k$ occur only for $i=k+1$.
\end{proof}

For $n=3$ this reads $\Imm_{(2,1)}=2\Imm_{213}+2\Imm_{132}+\Imm_{231}+\Imm_{312}$. As a check, Lemma~\ref{lem:J}(3) gives $\Imm_w(J_{(2,1)})=\delta_{w,213}$, so the right side is $2$ at $J_{(2,1)}$, in agreement with $\Imm_{(2,1)}(J_{(2,1)})=2!\,1!$ from Lemma~\ref{lem:J}(2).

\subsection{Rotations}
For $a\ge3$ the alternating sum \eqref{eq:cycle} does not collapse in the same way. Instead we prove a stronger statement, in which the per-vertex function is composed with a rotation, by induction on $a$. For $b\ge1$, $m\ge1$ and $v\in[m]$ write $F^{(b)}_v=f^{(b)}_v$ for the per-vertex function on $S_m$. Permutations are maps from positions to values, and $(\rho\pi)(x)=\rho(\pi(x))$.

\begin{definition}\label{def:good}
A permutation $\rho\in S_m$ is a \emph{good rotation for $v$} if $\rho=e$, or $\rho$ rotates a block $[l,h]$ of consecutive values with $v$ at its leading end: either $v=h$, $\rho(l)=h$ and $\rho(x)=x-1$ for $l<x\le h$ (a \emph{down-rotation}), or $v=l$, $\rho(h)=l$ and $\rho(x)=x+1$ for $l\le x<h$ (an \emph{up-rotation}). The other end of the block is the \emph{anchor}; for $\rho=e$ the anchor is $v$. A block of length $1$ gives $\rho=e$, and then either description has anchor $v$.
\end{definition}

\begin{theorem}\label{thm:rotated}
For every $b\ge1$, $m$, $v\in[m]$ and good rotation $\rho$ for $v$, the function $\pi\mapsto F^{(b)}_v(\rho\pi)$ satisfies $F^{(b)}_v(\rho\,\Cp u)\ge0$ for every $u\in S_m$ with $u\ne e$.
\end{theorem}

\begin{proof}[Proof of Theorem~\ref{thm:main}, given Theorem~\ref{thm:rotated}]
Take $\rho=e$: $f^{(a)}_i(\Cp u)\ge0$ for $u\ne e$. At $u=e$, $f^{(a)}_i(\Cp e)=f^{(a)}_i(e)=0$ for $a\ge2$, while $f^{(1)}_i=\sgn$ is KL-positive. Sum over $i$ with \eqref{eq:stembridge}.
\end{proof}

The proof of Theorem~\ref{thm:rotated} is an induction on $b$. Fix $b,m,v,\rho$ and put $h(\pi)=F^{(b)}_v(\rho\pi)$. For $c\in[m]$, the permutations with $\pi(v)=c$ form a coset of the stabiliser of the value $c$. As in Section~\ref{sec:prod}, let $\pi'\in S_{m-1}$ be $\pi$ with position $v$ and value $c$ deleted, standardised, and define $H_c(\pi')=h(\pi)$.

\begin{lemma}[closure]\label{lem:closure}
Let $t=\rho(c)$. If $t=v$, that is, if $c$ is the anchor, then $H_c=0$ (for $b\ge2$). Otherwise $H_c(\pi')=F^{(b-1)}_{v'}(\rho'\pi')$, where $v'$ is the position of $t$ in $[m]\setminus\{v\}$ and $\rho'$ is a good rotation for $v'$ whose block has length $|c-\text{anchor}|$.
\end{lemma}

\begin{proof}
Put $\sigma=\rho\pi$, so $\sigma(v)=t$. If $t=v$, the cycle of $\sigma$ through $v$ has length $1<b$. Otherwise remove $v$ from its cycle: let $\tau$ be the permutation of $[m]\setminus\{v\}$ with $\tau(x)=\sigma(x)$, except $\tau(\sigma^{-1}(v))=t$. The cycle through $t$ becomes one shorter, the other cycles are unchanged, and the number of cycles is the same, so $\sgn\sigma=-\sgn\tau$ and $F^{(b)}_v(\sigma)=F^{(b-1)}_t(\tau)$. For $x\ne v$, $\tau(x)=\psi(\rho(\pi(x)))$ with $\psi(v)=t$ and $\psi$ the identity elsewhere. Here $\psi\colon[m]\setminus\{t\}\to[m]\setminus\{v\}$ is a bijection because $t\ne v$. Let $\mathrm{std}_v\colon[m]\setminus\{v\}\to[m-1]$ and $\mathrm{std}_c\colon[m]\setminus\{c\}\to[m-1]$ be the order-preserving bijections, so that $\pi'=\mathrm{std}_c\circ\pi\circ\mathrm{std}_v^{-1}$. Conjugating $\tau$ by $\mathrm{std}_v$ gives $\mathrm{std}_v\circ\tau\circ\mathrm{std}_v^{-1}=\rho'\pi'$ with $\rho'=\mathrm{std}_v\circ\psi\circ\rho\circ\mathrm{std}_c^{-1}\in S_{m-1}$, and conjugation preserves cycle lengths and signs, so $F^{(b-1)}_t(\tau)=F^{(b-1)}_{v'}(\rho'\pi')$ with $v'=\mathrm{std}_v(t)$. By the symmetry $w\mapsto w_0ww_0$, which preserves Kazhdan--Lusztig polynomials, we may assume that $\rho=e$ or that $\rho$ is the down-rotation of $[l,v]$. A direct check gives three cases: if $c>v$, $\rho'$ is the down-rotation of $[l,c-1]$ and $v'=c-1$; if $l<c\le v$, $\rho'$ is the down-rotation of $[l,c-1]$ and $v'=c-1$; if $c<l$, $\rho'$ is the up-rotation of $[c,l-1]$ and $v'=c$. In each case $\rho'$ is good for $v'$ and its block has length $|c-l|$.
\end{proof}

\begin{exa}\label{exa:closure}
Let $m=6$ and let $\rho=125346$ be the down-rotation of $[3,5]$, so $v=5$ and the anchor is $l=3$. The three cases of the proof, and the anchor, give:
\begin{center}
\begin{tabular}{cccccc}
\toprule
$c$ & $t=\rho(c)$ & case & $\rho'\in S_5$ & $v'$ & $|c-l|$\\
\midrule
$1$ & $1$ & $c<l$ & $21345$, up-rotation of $[1,2]$ & $1$ & $2$\\
$2$ & $2$ & $c<l$ & $e$, up-rotation of $[2,2]$ & $2$ & $1$\\
$3$ & $5$ & anchor & $H_3=0$ & & \\
$4$ & $3$ & $l<c\le v$ & $e$, down-rotation of $[3,3]$ & $3$ & $1$\\
$5$ & $4$ & $l<c\le v$ & $12435$, down-rotation of $[3,4]$ & $4$ & $2$\\
$6$ & $6$ & $c>v$ & $12534$, down-rotation of $[3,5]$ & $5$ & $3$\\
\bottomrule
\end{tabular}
\end{center}
For instance, at $c=6$ the value $t=6$ is fixed by $\rho$, and removing $v=5$ from the cycle of $\sigma=\rho\pi$ through it gives a permutation $\tau$ of $\{1,2,3,4,6\}$; after standardising, $\tau=\rho'\pi'$ with $\rho'$ again the down-rotation of $[3,5]$, now in $S_5$, and $v'=5$ is the position of $t=6$. In each row $H_c=F^{(b-1)}_{v'}\circ\rho'$ has been checked on all of $S_5$ for $b=2,3,4,5$.
\end{exa}

\begin{lemma}[bottom terms]\label{lem:bottom}
For $u\ne e$, $\sum_cH_c(e)\,D_{vc}(u)\ge0$, where $D_{vc}(u)$ is the coefficient of $\Imm_u$ in $x_{vc}\det x[\hat v,\hat c]$.
\end{lemma}

\begin{proof}
Let $l$ be the anchor and let $c\in[m]$. If $c=l$, then $H_c=0$ by Lemma~\ref{lem:closure}. Otherwise $\rho'$ rotates a block of length $L=|c-l|\ge1$, so it is a single $L$-cycle through $v'$ (a fixed point if $L=1$) and $\sgn\rho'=(-1)^{L-1}$. Hence, in both cases,
\[
H_c(e)=F^{(b-1)}_{v'}(\rho')=(-1)^{b+|c-l|+1}\,[\,|c-l|\ge b-1\,].
\]
So $H_c(e)=0$ exactly for $c$ in $Z\cap[m]$, where $Z=[l-b+2,\,l+b-2]$ is an interval containing $l$ (as $b\ge2$), and
\[
H_c(e)=\kappa\,(-1)^{v+c}\quad(c\in[m]\setminus Z),\qquad \kappa=(-1)^{b+1+v+l}.
\]
By Proposition~\ref{prop:colour}, the coefficient of the diagram with cup $\{v,v\pm1\}$ and cap $\{k',(k+1)'\}$ in $\sum_cH_c(e)\,x_{vc}\det x[\hat v,\hat c]$ is $H_k(e)+H_{k+1}(e)$ (Figure~\ref{fig:bottom}). If $k$ and $k+1$ both lie outside $Z$, this is $\kappa(-1)^{v+k}+\kappa(-1)^{v+k+1}=0$, and if both lie in $Z$ it is $0$. Since $Z$ is an interval, the remaining pairs straddle its boundary. For $k=l-b+1$, with $k+1\in Z$, the coefficient is $H_k(e)=\kappa(-1)^{v+l-b+1}=+1$. For $k+1=l+b-1$, with $k\in Z$, it is $H_{k+1}(e)=\kappa(-1)^{v+l+b-1}=+1$. A boundary pair that does not lie inside $[m]$ gives no diagram. So
\[
\sum_cH_c(e)\,x_{vc}\det x[\hat v,\hat c]\;=\;H_v(e)\,\Imm_e\;+\sum_{\substack{k\in\{l-b+1,\;l+b-2\}\\1\le k<m}}\ \sum_{\pm}\Imm_{\tau_{\pm,k}},
\]
where $\tau_{\pm,k}$ has cup $\{v,v\pm1\}$ and cap $\{k',(k+1)'\}$, and the inner sum is over the cups that exist. The two boundary values $k=l-b+1$ and $k=l+b-2$ differ, since $2b\ne3$, so the diagrams in the double sum have distinct caps or distinct cups and are pairwise distinct by Proposition~\ref{prop:colour}. Distinct diagrams correspond to distinct $321$-avoiding permutations \cite{FG}, hence to distinct Kazhdan--Lusztig immanants, so for $u\ne e$ the sum $\sum_cH_c(e)D_{vc}(u)$ is $0$ or $1$.
\end{proof}

\begin{figure}[ht]
\centering
\begin{tikzpicture}[scale=0.62,
  blk/.style={circle,fill=black,inner sep=1.6pt},
  wht/.style={circle,draw=black,fill=white,inner sep=1.6pt}]
\foreach \pan/\c/\lab in {0/4/{c=k=4},7.5/5/{c=k+1=5}}{
  \begin{scope}[xshift=\pan cm]
    \draw[thick] (0,-2) .. controls (1.1,-2) and (1.1,-3) .. (0,-3);
    \draw[thick] (3,-4) .. controls (1.9,-4) and (1.9,-5) .. (3,-5);
    \draw[thick] (0,-1)--(3,-1); \draw[thick] (0,-4)--(3,-2); \draw[thick] (0,-5)--(3,-3); \draw[thick] (0,-6)--(3,-6);
    \foreach \i in {1,...,6}{
      \ifnum\i=2 \node[blk] at (0,-\i) {}; \else \node[wht] at (0,-\i) {}; \fi
      \ifnum\i=\c \node[wht] at (3,-\i) {}; \else \node[blk] at (3,-\i) {}; \fi
      \node[left,font=\scriptsize] at (-0.15,-\i) {$\i$};
      \node[right,font=\scriptsize] at (3.15,-\i) {$\i'$};
    }
    \node[font=\small] at (1.5,-7) {$x_{2c}\det x[\hat2,\hat c]$, $\lab$};
  \end{scope}}
\begin{scope}[yshift=-8.6cm,xshift=1.8cm]
  \node[font=\small,left] at (0.5,0) {$c$};
  \node[font=\small,left] at (0.5,-0.8) {$H_c(e)$};
  \node[font=\small,left] at (0.5,-1.6) {$H_k(e)+H_{k+1}(e)$};
  \foreach \c/\h in {1/0,2/0,3/0,4/+1,5/-1,6/+1}{
    \node[font=\small] at (1.2*\c,0) {$\c$};
    \node[font=\small] at (1.2*\c,-0.8) {$\h$};}
  \foreach \k/\t in {1/0,2/0,3/+1,4/0,5/0}{
    \node[font=\small] at (1.2*\k+0.6,-1.6) {$\t$};}
  \draw[thick] (0.9,0.35) -- (0.9,0.45) -- (3.9,0.45) -- (3.9,0.35);
  \node[font=\scriptsize,above] at (2.4,0.45) {$Z=[1,3]$};
\end{scope}
\end{tikzpicture}
\caption{Lemma~\ref{lem:bottom} for $m=6$, $v=2$, $k=4$. The diagram with cup $\{2,3\}$ and cap $\{4',5'\}$ is compatible with the colouring of $x_{24}\det x[\hat2,\hat4]$ (left) and with that of $x_{25}\det x[\hat2,\hat5]$ (right), and with no other summand; black vertices are filled. Its two coefficients $H_4(e)$ and $H_5(e)$ have opposite signs unless one of $4,5$ lies in $Z$. Below, the values $H_c(e)$ for $\rho=e$ and $b=3$, where $l=v=2$ and $Z=[1,3]$. The last row, placed between $k$ and $k+1$, is the coefficient of the diagram with cap $\{k',(k+1)'\}$: the pairs $(4,5)$ and $(5,6)$ cancel, the only surviving pair is the boundary pair $(3,4)$, with coefficient $H_4(e)=+1$, and the other boundary $k=l-b+1=0$ lies outside $[m]$.}\label{fig:bottom}
\end{figure}
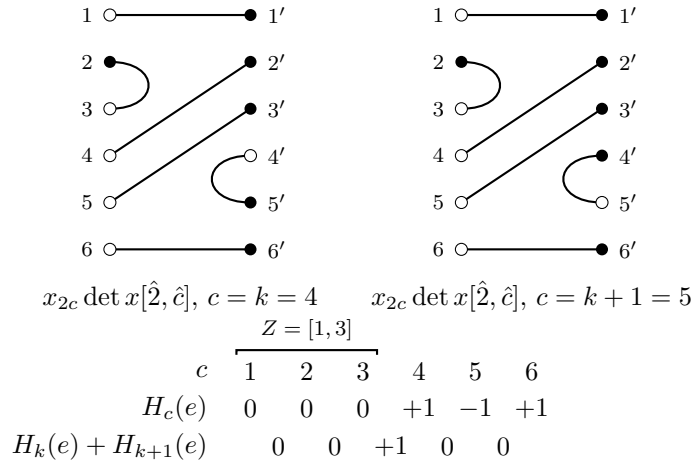

\begin{proof}[Proof of Theorem~\ref{thm:rotated}]
Induction on $b$. For $b=1$, $F^{(1)}_v(\rho\pi)=\sgn(\rho)\sgn(\pi)$, and $\sgn(\Cp u)=\delta_{u,e}$. Let $b\ge2$. Apply \eqref{eq:BB} with $k=1$ to the coset $\{\pi(v)=c\}$: for each $c$ there are integers $a^{(c)}_y(u)\ge0$, indexed by the coset, with
\[
h(\Cp u)=\sum_c\sum_{y}a^{(c)}_y(u)\,H_c\bigl(\Cp{\varphi(y)}\bigr).
\]
The terms with $\varphi(y)\ne e$ are $\ge0$ by Lemma~\ref{lem:closure} and induction. Each coset contains exactly one $y$ with $\varphi(y)=e$, since $\varphi$ is a bijection of the coset onto $S_{m-1}$, and its term is $a^{(c)}_y(u)H_c(e)$. The same multiplicity occurs in Theorem~\ref{thm:R2} with $I=\{v\}$, $T=\{c\}$, $F=x_{vc}$ and $G=\det x[\hat v,\hat c]=\Imm_e(x[\hat v,\hat c])$: there $x_{vc}\det x[\hat v,\hat c]=d_g$ with $g(\sigma)=[\sigma(v)=c]\,\sgn(\varphi(\sigma))$, and since $\sgn(\Cp{u'})=\delta_{u',e}$, the proof of Theorem~\ref{thm:R2} shows that the coefficient of $\Imm_u$ is $a^{(c)}_y(u)$ for this $y$. So $a^{(c)}_y(u)=D_{vc}(u)$, and these terms sum to $\sum_cH_c(e)D_{vc}(u)\ge0$ by Lemma~\ref{lem:bottom}.
\end{proof}

For $\rho=e$ the bottom terms are $(-1)^{a-1}$ times a tail $\sum_{|j-i|\ge a-1}(-1)^{i+j}x_{ij}\det x[\hat\imath,\hat\jmath]$ of the Laplace expansion of $\det x$ along row $i$, and Lemma~\ref{lem:bottom} identifies them with at most four staircase Kazhdan--Lusztig immanants. For $a=2$ this is Lemma~\ref{lem:a2}.

\begin{exa}\label{ex:n4}
Let $n=4$, $a=3$, $i=1$. The non-zero Kazhdan--Lusztig coefficients of $f^{(3)}_1$ are $1$ at $2314$, $3124$, $3241$, $4213$ and $2$ at $3214$, $4231$. In the proof, $\rho=e$, $v=l=1$ and $b=3$, so $H_c(e)=0$ for $c\in\{1,2\}$, $H_3(e)=+1$ and $H_4(e)=-1$. Only the boundary cap $\{2',3'\}$ survives, with the cup $\{1,2\}$, and this diagram is the staircase $3124$. So the bottom terms contribute exactly $\Imm_{3124}$. The rest of $f^{(3)}_1$, namely $1$ at $2314$, $3241$, $4213$ and $2$ at $3214$, $4231$, comes from the pieces $H_c=F^{(2)}_{v'}\circ\rho'$, which are non-negative off $e$ by the case $b=2$. In detail, $H_1=0$, and for $c=2,3,4$ the closure lemma gives $(\rho',v')=(e,1)$, $(213,2)$ and $(312,3)$. Writing $D^{(c)}_{u'}(u)$ for the coefficient of $\Imm_u$ in $x_{1c}\Imm_{u'}(x[\hat1,\hat c])$, the proof of Theorem~\ref{thm:rotated} splits $h(\Cp u)=\sum_c\sum_{u'}D^{(c)}_{u'}(u)\,H_c(\Cp{u'})$ as follows:
\begin{center}
\begin{tabular}{cll}
\toprule
$c$ & $u'$ with $H_c(\Cp{u'})\ne0$ & $x_{1c}\Imm_{u'}(x[\hat1,\hat c])$\\
\midrule
$2$ & $213$: $1$ & $\Imm_{2314}+\Imm_{3214}$\\
$3$ & $e$: $1$ & $\Imm_{3124}+\Imm_{4123}$\\
 & $213$: $1$ & $\Imm_{3214}+\Imm_{4213}$\\
 & $231$: $1$ & $\Imm_{3241}+\Imm_{4231}$\\
$4$ & $e$: $-1$ & $\Imm_{4123}$\\
 & $231$: $1$ & $\Imm_{4231}$\\
\bottomrule
\end{tabular}
\end{center}
So the contributions of $c=2,3,4$ are
\begin{gather*}
\Imm_{2314}+\Imm_{3214},\qquad \Imm_{3124}+\Imm_{4123}+\Imm_{3214}+\Imm_{4213}+\Imm_{3241}+\Imm_{4231},\\
-\Imm_{4123}+\Imm_{4231}.
\end{gather*}
The two bottom terms at $\Imm_{4123}$ cancel, as in Lemma~\ref{lem:bottom}, and the three contributions add up to the expansion of $f^{(3)}_1$ above.
\end{exa}

The induction works at $q=1$ only, because \eqref{eq:BB} is a statement about Euler characteristics. We do not know whether $\phi^{(a,1^{n-a})}_q(q^{\ell(w)/2}C'_w(q))\in\N[q]$ for every $n$ (Problem~\ref{prob:q}); for $n\le10$ it holds by \cite{McN}.

\section{Consequences}\label{sec:consequences}

\begin{proof}[Proof of Corollary~\ref{cor:semiring}]
The KL-positive class functions are closed under sums, and under induction products by Theorem~\ref{thm:R2}(3). Every irreducible character is KL-positive by Haiman \cite[Lemma~1.1]{Hai}, and every hook monomial character by Theorem~\ref{thm:main}. Since $\ch$ is additive and multiplicative, every class function with characteristic in $\mathcal R$ is KL-positive. The immanant is then non-negative on $\TNN_n$ by \eqref{eq:duality} and the positivity of the $\Imm_w$ \cite{RS06}.
\end{proof}

Corollary~\ref{cor:semiring} contains the $h_\mu=s_{(\mu_1)}s_{(\mu_2)}\cdots$, so every $\eta^\mu$ is KL-positive, as is directly clear from $d_{\eta^\mu}(x)=\sum\prod_t\per x[T_t,T_t]$. A product such as $h_jm_{(a,1^b)}$ involves monomial symmetric functions $m_\lambda$ with $\lambda$ not a hook, so Corollary~\ref{cor:semiring} reaches class functions that are not combinations of hook monomial characters, and it does not follow from Haiman's theorem alone.

Write $f^a=\chi^a(e)$ and $N_a=\Imm_{\chi^a}/f^a$ for the normalised immanants.

\begin{corollary}\label{cor:skandera}
For every $n$ and $2\le k\le n$ put $D_k=(k-1)\chi^{(k,1^{n-k})}-(n-k+1)\chi^{(k-1,1^{n-k+1})}$, so that $N_{(k,1^{n-k})}-N_{(k-1,1^{n-k+1})}=\frac{(k-2)!\,(n-k)!}{(n-1)!}\,d_{D_k}$. Then
\[
\ch D_k=\sum_{j=0}^{k-2}h_j\,m_{(k-j,1^{n-k})},\qquad
d_{D_k}(x)=\sum_{j=0}^{k-2}\ \sum_{|T|=j}\per x[T,T]\;\Imm_{(k-j,1^{n-k})}\bigl(x[T^c,T^c]\bigr),
\]
and $D_k$ is KL-positive. In particular $\per\ge N_{(n-1,1)}\ge\cdots\ge\det$ on $\TNN_n$ for every $n$, with every difference a non-negative combination of Kazhdan--Lusztig immanants.
\end{corollary}

\begin{proof}
Let $H(x)=\sum_jh_jx^j$ and $E(y)=\sum_je_jy^j$ in variables $x_1,x_2,\dots$, and $F=\sum_{a,b\ge0}s_{(a+1,1^b)}x^ay^b$. The hook Schur functions satisfy $(x+y)F=H(x)E(y)-1$. The coefficient of $x^{k-1}y^{n-k}$ in $x(\partial_x-\partial_y)F$ is $\ch D_k$, and since $(\partial_x-\partial_y)(x+y)=0$,
\[
x(\partial_x-\partial_y)F=\frac{x\,(H'E-HE')}{x+y}=xHE\sum_i\frac{x_i^2}{(1-xx_i)(1+yx_i)},
\]
using $H'/H=\sum_ix_i/(1-xx_i)$ and $E'/E=\sum_ix_i/(1+yx_i)$. On the other hand $\sum_bm_{(c,1^b)}y^b=\sum_ix_i^c\prod_{j\ne i}(1+yx_j)$ for $c\ge2$, so $H(x)\sum_{c\ge2,\,b\ge0}m_{(c,1^b)}x^{c-1}y^b$ is the same series. Comparing coefficients of $x^{k-1}y^{n-k}$ gives the first identity. The second follows from \eqref{eq:indprod}. KL-positivity is Corollary~\ref{cor:semiring}.
\end{proof}

Skandera's proofs \cite{Ska25} give total non-negativity: one counts tableaux on planar networks, and the other expands $D_k$ in the immanants of $\theta_\ell=\sum_{\ell(\mu)=\ell}\phi^\mu$, whose KL-positivity is open (Problem~\ref{prob:theta}). Corollary~\ref{cor:skandera} avoids the $\theta_\ell$. For total non-negativity alone, the theorem of Dong--Wu--Gao \cite{DWG} can replace Theorem~\ref{thm:main} in the proof.

The power sums $p_k=m_{(k)}$ are hook monomial symmetric functions, so Corollary~\ref{cor:semiring} contains every $s_\mu p_k$. The Murnaghan--Nakayama rule expands these in Schur functions. For a border strip $\lambda/\mu$ let $\mathrm{ht}(\lambda/\mu)$ be the number of its rows minus one.

\begin{corollary}\label{cor:MN}
Let $k\ge1$ and $\mu\vdash n-k$, and let $\psi_{\mu,k}=\sum_\lambda(-1)^{\mathrm{ht}(\lambda/\mu)}\chi^\lambda$, the sum over all $\lambda\vdash n$ such that $\lambda/\mu$ is a border strip with $k$ boxes. Then $\psi_{\mu,k}$ is KL-positive, so
\[
\sum_\lambda(-1)^{\mathrm{ht}(\lambda/\mu)}\Imm_{\chi^\lambda}\ge0\quad\text{on }\TNN_n.
\]
In particular, let $r\ge1$, $n=\binom{r+1}2+2$ and $\lambda_r=(r+2,r-1,r-2,\dots,1)$. Then $N_{\lambda_r}\ge N_{\lambda_r'}$ on $\TNN_n$, and $f^{\lambda_r}(N_{\lambda_r}-N_{\lambda_r'})$ is a non-negative combination of Kazhdan--Lusztig immanants.
\end{corollary}

\begin{proof}
By the Murnaghan--Nakayama rule $\ch\psi_{\mu,k}=s_\mu p_k$, so Corollary~\ref{cor:semiring} applies. Theorem~\ref{thm:main} is not needed here: $p_k=\ch\phi^{(k)}$, where $\phi^{(k)}(\sigma)=k$ if $\sigma$ is a $k$-cycle and $\phi^{(k)}(\sigma)=0$ otherwise. This is a non-negative function, and $\Cp w$ has non-negative coefficients, so $\phi^{(k)}$ is KL-positive. Then $\psi_{\mu,k}=\chi^\mu\cdot\phi^{(k)}$ is KL-positive by Haiman \cite[Lemma~1.1]{Hai} and Theorem~\ref{thm:prod}.

For the staircase $\delta_r=(r,r-1,\dots,1)$ and $k=2$, the border strips are dominoes. Consecutive rows and consecutive columns of $\delta_r$ differ in length by one, so $\delta_r$ has exactly two addable dominoes: a horizontal one in the first row, giving $\lambda_r$ with sign $+1$, and a vertical one in the first column, giving $(r,r-1,\dots,2,1,1,1)=\lambda_r'$ with sign $-1$. Hence $\psi_{\delta_r,2}=\chi^{\lambda_r}-\chi^{\lambda_r'}$. Conjugate shapes have the same dimension, so $d_{\psi_{\delta_r,2}}=f^{\lambda_r}(N_{\lambda_r}-N_{\lambda_r'})$.
\end{proof}

The smallest case with three terms is $\mu=(1,1)$, $k=2$: $s_{11}p_2=s_{31}-s_{22}-s_{1111}$, so $\Imm_{\chi^{(3,1)}}\ge\Imm_{\chi^{(2,2)}}+\det$ on $\TNN_4$. The staircase inequalities for $r=1,2$ are $\per\ge\det$ on $\TNN_3$ and $N_{(4,1)}\ge N_{(2,1,1,1)}$, a composite of two steps of Corollary~\ref{cor:skandera}. For $r=3$ the inequality is $N_{(5,2,1)}\ge N_{(3,2,1,1,1)}$ on $\TNN_8$, whose larger side is not a hook. More generally, Corollary~\ref{cor:semiring} makes every $s_\mu m_{(a,1^b)}$ KL-positive; for $b\ge1$ this does use Theorem~\ref{thm:main}.

Total non-negativity in Corollary~\ref{cor:MN} also follows directly from \eqref{eq:indprod}: principal submatrices of a totally non-negative matrix are totally non-negative, $\Imm_{\chi^\mu}\ge0$ there by Stembridge \cite{Ste91}, and for a $k\times k$ matrix $y$, $d_{\phi^{(k)}}(y)=k\sum_\sigma\prod_iy_{i\sigma(i)}$, summed over the $k$-cycles $\sigma$, is non-negative whenever $y$ has non-negative entries. The new content is the Kazhdan--Lusztig statement. For $k=2$ the inequalities also hold on Hermitian positive semidefinite matrices, by the same argument with Schur's theorem \cite{Sch18} in place of Stembridge's, since $d_{\phi^{(2)}}(y)=2|y_{12}|^2$ for a Hermitian $2\times2$ matrix $y$.

\section{Three cones}\label{sec:cones}

In this section we set up the cones $\CKL\subseteq\CTNN\subseteq\CJ$ of the introduction (Figure~\ref{fig:cones}), and show how they organise the known results. Let $w_0(\alpha)$ be the longest element of $S_\alpha$.

\begin{lemma}\label{lem:J}
\begin{enumerate}
\item For every composition $\alpha$ of $n$, $J_\alpha\in\TNN_n$.
\item For every class function $\phi$, $d_\phi(J_\mu)=\mu!\,\langle\phi,\eta^\mu\rangle$. In particular $\Imm_\lambda(J_\mu)=\mu!\,\delta_{\lambda\mu}$.
\item For every $w\in S_n$, $\Imm_w(J_\alpha)=\delta_{w,w_0(\alpha)}$. Hence $d_\phi(J_\mu)=\phi(\Cp{w_0(\mu)})$.
\end{enumerate}
\end{lemma}

\begin{proof}
(1) Order the rows and columns of a square submatrix by the blocks they lie in. If some block contributes more rows than columns, those rows are supported on too few columns and the minor vanishes. Otherwise every block contributes equally many rows and columns, since the totals agree, so the submatrix is block diagonal with square all-ones blocks, and the minor is a product of determinants of all-ones matrices, each equal to $0$ or $1$.
(2) $\prod_i(J_\mu)_{i,\sigma(i)}=[\sigma\in S_\mu]$, so $d_\phi(J_\mu)=\sum_{\sigma\in S_\mu}\phi(\sigma)=\mu!\,\langle\phi,\eta^\mu\rangle$ by Frobenius reciprocity.
(3) The lower interval of $w_0(\alpha)$ is $S_\alpha$, and $P_{v,w_0(\alpha)}=1$ for $v\in S_\alpha$, so $\Cp{w_0(\alpha)}=\sum_{v\in S_\alpha}v$. For every $f$, \eqref{eq:duality} at $J_\alpha$ gives $\sum_wf(\Cp w)\Imm_w(J_\alpha)=\sum_{v\in S_\alpha}f(v)=f(\Cp{w_0(\alpha)})$. Since $\{\Cp w\}$ is a basis and $f$ is arbitrary, $\Imm_w(J_\alpha)=\delta_{w,w_0(\alpha)}$.
\end{proof}

\begin{theorem}\label{thm:cones}
For every $n$, $\CKL\subseteq\CTNN\subseteq\CJ$, and $\CJ=\operatorname{cone}\{\phi^\lambda:\lambda\vdash n\}$ is simplicial. Consequently:
\begin{enumerate}
\item Conjecture~\ref{conj:MN} holds for $n$ iff $\CTNN=\CJ$, iff every class function $\phi$ with $d_\phi(J_\mu)\ge0$ for all $\mu$ satisfies $d_\phi\ge0$ on $\TNN_n$;
\item Conjecture~\ref{conj:H} holds for $n$ iff $\CKL=\CJ$, iff KL-positivity of a class function is detected by its $p(n)$ coefficients at the elements $w_0(\mu)$;
\item Conjecture~\ref{conj:H} holds for $n$ iff $\CKL=\CTNN=\CJ$; in particular it implies Conjecture~\ref{conj:MN} and $\CKL=\CTNN$.
\end{enumerate}
\end{theorem}

\begin{proof}
The first inclusion is \eqref{eq:duality} together with $\Imm_w\ge0$ on $\TNN_n$, and the second is Lemma~\ref{lem:J}(1). By Lemma~\ref{lem:J}(2), $\phi\in\CJ$ iff $\langle\phi,\eta^\mu\rangle\ge0$ for all $\mu$, iff $\phi$ has non-negative coordinates in the basis $\{\phi^\lambda\}$ dual to $\{\eta^\mu\}$. So $\CTNN=\CJ$ iff every generator $\phi^\lambda$ lies in $\CTNN$, which is Conjecture~\ref{conj:MN}, and similarly for $\CKL$. The reformulation in (2) is Lemma~\ref{lem:J}(3). Part~(3) follows from the chain of inclusions.
\end{proof}

The inclusion $\CKL\subseteq\CJ$, by evaluation at the longest elements $w_0(\mu)$ of the Young subgroups, is due to Haiman, who observed that any class function with the positivity of Conjecture~\ref{conj:H} is a non-negative combination of monomial characters \cite[\S2 and Prop.~4.1(1)]{Hai}.

\emph{The dual picture.} For $h\in\Z[S_n]$ put $\Phi(h)=\sum_\lambda\phi^\lambda(h)\,e_\lambda$. Since $\chi^a=\sum_\lambda K_{a,\lambda}\phi^\lambda$ and $h_\lambda=\sum_aK_{a,\lambda}s_a$, we have $\sum_a\chi^a(h)\,s_a=\sum_\lambda\phi^\lambda(h)\,h_\lambda$, so $\Phi(h)=\omega\bigl(\sum_a\chi^a(h)\,s_a\bigr)$, the image under $\omega$ of the characteristic $\ch(h)$ used in Section~\ref{sec:clean}. If $\phi=\sum_\lambda\gamma_\lambda\phi^\lambda$, then $\phi(h)=\ell_\gamma(\Phi(h))$ for the linear functional $\ell_\gamma(e_\lambda)=\gamma_\lambda$. By Lemma~\ref{lem:J}(2),(3), $\Phi(\Cp{w_0(\mu)})=\mu!\,e_\mu$. So $\CJ$ consists of the $\phi$ with $\phi(h)\ge0$ whenever $\Phi(h)$ is $e$-positive, $\CKL$ is dual to the cone spanned by the $\Phi(\Cp w)$, and Conjecture~\ref{conj:H} says that every $\Phi(\Cp w)$ is $e$-positive. At codominant $w$, $\Phi(\Cp w)$ is Stanley's chromatic symmetric function $X_P$ of a unit interval order \cite[Cor.~7.2]{CHSS}, so there Conjecture~\ref{conj:H} is the Stanley--Stembridge conjecture, proved by Hikita \cite{Hik}; smooth $w$ reduce to codominant ones (Proposition~\ref{thm:smooth}). Haiman had already noted that at codominant $w$ Conjecture~\ref{conj:H} is equivalent to a case of Stembridge's conjecture that the monomial immanants of Jacobi--Trudi matrices are Schur-positive \cite[\S5]{Hai}, \cite{Ste92}.

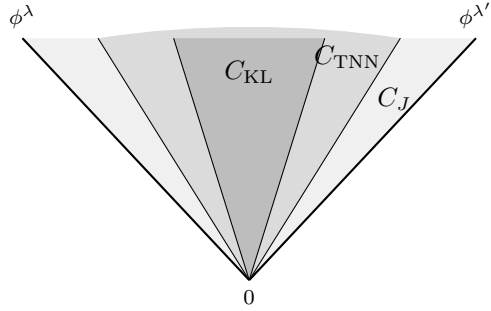
\begin{figure}[ht]
\centering
\begin{tikzpicture}[scale=1.0]
  \coordinate (O) at (0,0);
  \fill[black!6] (O) -- (-3,3.2) -- (3,3.2) -- cycle;
  \draw[thick] (O) -- (-3,3.2) (O) -- (3,3.2);
  \fill[black!14] (O) -- (-2,3.2) .. controls (-0.7,3.4) and (0.7,3.4) .. (2,3.2) -- cycle;
  \draw (O) -- (-2,3.2) (O) -- (2,3.2);
  \fill[black!26] (O) -- (-1,3.2) -- (1,3.2) -- cycle;
  \draw (O) -- (-1,3.2) (O) -- (1,3.2);
  \node[font=\small] at (0,2.75) {$\CKL$};
  \node[font=\small] at (1.3,2.95) {$\CTNN$};
  \node[font=\small] at (1.92,2.4) {$\CJ$};
  \node[font=\scriptsize,above] at (-3,3.2) {$\phi^\lambda$};
  \node[font=\scriptsize,above] at (3,3.2) {$\phi^{\lambda'}$};
  \node[font=\scriptsize,below] at (O) {$0$};
\end{tikzpicture}
\caption{A schematic slice of the three cones of Theorem~\ref{thm:cones}. The outer cone $\CJ$ is simplicial, with extreme rays $\phi^\lambda$, $\lambda\vdash n$. Conjecture~\ref{conj:MN} says that the middle cone fills it, and Conjecture~\ref{conj:H} that the inner one does. For $n\le10$ all three coincide \cite{McN}, so the picture shows only the logical possibilities, not an actual $n$.}\label{fig:cones}
\end{figure}

\subsection{What is known}
In this language, the known general results read as follows.
\begin{itemize}
\item Rhoades--Skandera \cite{RS06}: $\Imm_w\ge0$ on $\TNN_n$, which gives $\CKL\subseteq\CTNN$.
\item Haiman \cite[Lemma~1.1]{Hai}: every irreducible character lies in $\CKL$. This refines Stembridge's theorem \cite{Ste91} that $\chi^a\in\CTNN$.
\item Clearman--Shelton--Skandera \cite{CSS}: $\phi^\lambda\in\CKL$ when $\lambda_1\le2$. Stembridge \cite[Thms.~2.7, 2.8]{Ste92}: $\phi^\lambda\in\CTNN$ for rectangles and $(2,1^{n-2})$. Dong--Wu--Gao \cite{DWG}: $\phi^\lambda\in\CTNN$ for hooks (indeed under the weaker hypothesis of Stembridge's \cite[Question~2.9]{Ste92}).
\item This paper: $\phi^\lambda\in\CKL$ for hooks (Theorem~\ref{thm:main}); $\CKL$ is closed under induction products (Theorem~\ref{thm:R2}), and so are $\CTNN$, by \eqref{eq:indprod}, and $\CJ$, since products of monomial symmetric functions are monomial-positive. This is consistent with the three cones being equal.
\item The companion paper \cite{McN}: Haiman's conjecture, at generic $q$, for $n\le10$; hence $\CKL=\CTNN=\CJ$ for $n\le10$.
\end{itemize}
Since each $\phi^\lambda$ spans an extreme ray of $\CJ\supseteq\CKL$, a sum of elements of $\CKL$ is a multiple of $\phi^\lambda$ only if every summand is; and an induction product of two non-zero elements of $\CJ$ is never a multiple of a single $\phi^\lambda$, since a product $m_\alpha m_\beta$ of non-empty monomial symmetric functions contains both $m_{\alpha\cup\beta}$ and $m_{\alpha+\beta}$. So sums and induction products of KL-positive class functions never prove $\phi^\lambda\in\CKL$ for a new $\lambda$, although they do produce new elements of $\CKL$, as in Corollary~\ref{cor:skandera}. A proof for a new shape must leave the class functions, as the proof of Theorem~\ref{thm:main} does with the per-vertex functions $f^{(a)}_i$, which are not class functions; off the class functions $\CKL$ and $\CTNN$ already differ for $n=4$ (Section~\ref{sec:beyond}).

\subsection{Inequalities between normalised immanants}
For $a,b\vdash n$ write $a\succeq b$ if $N_a\ge N_b$ on $\TNN_n$. The cones turn this order into a question about Kostka numbers.

\begin{proposition}\label{prop:ordercrit}
Let $a,b\vdash n$. If $a\succeq b$ then $K_{a,\mu}/f^a\ge K_{b,\mu}/f^b$ for all $\mu\vdash n$, and if this criterion fails at $\mu$ then $N_a(J_\mu)<N_b(J_\mu)$. Conversely, if Conjecture~\ref{conj:MN} holds for $n$ and the criterion holds, then $a\succeq b$. In particular, $a\succeq b$ implies that $a$ dominates $b$.
\end{proposition}

\begin{proof}
$f^af^b(N_a-N_b)=d_\phi$ with $\phi=f^b\chi^a-f^a\chi^b$, and $d_\phi(J_\mu)=\mu!\,(f^bK_{a,\mu}-f^aK_{b,\mu})$ by Lemma~\ref{lem:J}. The first statement holds because $J_\mu\in\TNN_n$, and the converse is Theorem~\ref{thm:cones}(1). At $\mu=b$ the criterion gives $K_{a,b}\ge f^aK_{b,b}/f^b>0$, so $a$ dominates $b$.
\end{proof}

The first statement is Stembridge's \cite[Prop.~3.1]{Ste92}, and he notes that Conjecture~\ref{conj:MN} gives the converse \cite[\S3]{Ste92}.

\begin{proposition}\label{prop:topbottom}
For every $n$ and $a\vdash n$, $\per\ge N_a\ge\det$ on $\TNN_n$, and both differences are KL-positive.
\end{proposition}

\begin{proof}
The coefficient of $x^v$ in $\per-N_a$ is $1-\chi^a(v)/f^a\ge0$, and every monomial is KL-positive. For the bottom, $N_a-\det=\sum_{w\ne e}\frac{\chi^a(\Cp w)}{f^a}\Imm_w$, because $\Imm_e=\det$ and the coefficient of $\Imm_e$ is $\chi^a(e)/f^a-1=0$, and $\chi^a(\Cp w)\ge0$ by Haiman \cite[Lemma~1.1]{Hai}.
\end{proof}

\begin{proposition}\label{prop:psd}
If $d_\phi\ge0$ on all Hermitian positive semidefinite $n\times n$ matrices, then $\phi\in\CJ$. Hence $d_\phi\ge0$ on $\TNN_n$ if Conjecture~\ref{conj:MN} holds for $n$, and $d_\phi$ is KL-positive if Conjecture~\ref{conj:H} holds for $n$.
\end{proposition}

\begin{proof}
Each $J_\mu$ is a direct sum of matrices $\mathbf 1\mathbf 1^T$, so it is positive semidefinite. Apply Theorem~\ref{thm:cones}.
\end{proof}

For normalised characters, $\phi=f^b\chi^a-f^a\chi^b$, this is Stembridge's observation that, since each $J_\mu$ is positive semidefinite, an inequality $N_a\ge N_b$ on positive semidefinite matrices implies the criterion of Proposition~\ref{prop:ordercrit} \cite[\S3]{Ste92}. So every linear inequality between immanants known on positive semidefinite matrices transfers to $\TNN_n$ under Conjecture~\ref{conj:MN}, and to the Kazhdan--Lusztig level under Conjecture~\ref{conj:H}. Examples are Schur's inequality $N_a\ge\det$ \cite{Sch18}, Heyfron's hook chain \cite{Hey88}, and Pate's inequalities \cite{Pat91}, which compare $N_\lambda$ with $N_\mu$ when $\mu$ is obtained from $\lambda$ by removing its last column and adding its boxes to the bottom of the first column. Skandera asked whether Pate's inequalities hold on $\TNN_n$ \cite[Problem~29]{Ska25}; by Proposition~\ref{prop:psd} they follow from Conjecture~\ref{conj:MN}, and hence hold for $n\le10$ \cite{McN}. For every $n$, Corollary~\ref{cor:MN} gives one composite of two Pate steps at the Kazhdan--Lusztig level: $\lambda_r\to(r+1,r-1,\dots,1,1)\to\lambda_r'$, each step moving a last column of one box. For Schur's inequality and Heyfron's chain, Proposition~\ref{prop:topbottom} and Corollary~\ref{cor:skandera} give the Kazhdan--Lusztig statement for every $n$ unconditionally.

\section{Smooth and singular permutations}\label{sec:smoothsing}

\subsection{Smooth permutations}\label{sec:smooth}
A permutation is \emph{codominant} if it avoids $312$. For codominant $w$ put $m(i)=\max(w(1),\dots,w(i))$, a Hessenberg function, and let $A_w$ be the $0/1$ matrix with $(A_w)_{ij}=[j\le m(i)]$.

\begin{proposition}\label{prop:hessenberg}
For codominant $w$, $A_w\in\TNN_n$ and $\sum_{v}\prod_i(A_w)_{i,v(i)}\,v=\Cp w$. Consequently $f(\Cp w)=d_f(A_w)$ for every $f\colon S_n\to\R$, and $\Imm_v(A_w)=\delta_{v,w}$ for every $v$.
\end{proposition}

\begin{proof}
In a square submatrix of $A_w$, row $s$ is the indicator of an initial segment $[1,\mu_s]$ with $\mu$ non-decreasing. Subtracting consecutive rows leaves rows supported on disjoint intervals, so every minor is $0$ or $1$. Next, $\prod_i(A_w)_{i,v(i)}=[v(i)\le m(i)\ \forall i]$, which equals $[v\le w]$ for codominant $w$: a $312$-avoiding $w$ avoids the four patterns of \cite[Thm.~3.1]{Sjo}, so $[e,w]$ is the set of rook placements on its right hull, which is the Ferrers board $\{(i,j):j\le m(i)\}$ \cite[Rem.~3.4]{Sjo}. Codominant $w$ are smooth, so $P_{v,w}=1$ for all $v\le w$, and the sum is $\Cp w$. Comparing $d_f(A_w)=f(\Cp w)$ with \eqref{eq:duality} for all $f$ gives $\Imm_v(A_w)=\delta_{v,w}$.
\end{proof}

Proposition~\ref{prop:hessenberg} is essentially known: see the proof of \cite[Prop.~25]{Ska25} and the references there. For $w=w_0(\alpha)$ the matrix $A_w$ is not $J_\alpha$, since it has ones below the diagonal blocks, but both have $\prod_iM_{i,v(i)}=[v\in S_\alpha]$, so Proposition~\ref{prop:hessenberg} extends Lemma~\ref{lem:J}(3).

\begin{exa}\label{exa:hessenberg}
For the codominant $w=2341$, $m=(2,3,4,4)$, and for $w=w_0(2,2)=2143$, $m=(2,2,4,4)$:
\[
A_{2341}=\begin{pmatrix}1&1&0&0\\1&1&1&0\\1&1&1&1\\1&1&1&1\end{pmatrix},\qquad
A_{2143}=\begin{pmatrix}1&1&0&0\\1&1&0&0\\1&1&1&1\\1&1&1&1\end{pmatrix},\qquad
J_{(2,2)}=\begin{pmatrix}1&1&0&0\\1&1&0&0\\0&0&1&1\\0&0&1&1\end{pmatrix}.
\]
Each row of $A_w$ is the indicator of an initial segment, and the segments grow down the rows. One checks directly that $\Imm_v(A_{2341})=\delta_{v,2341}$ and $\Imm_v(A_{2143})=\Imm_v(J_{(2,2)})=\delta_{v,2143}$ for all $v\in S_4$.
\end{exa}

\begin{proposition}[Clearman--Hyatt--Shelton--Skandera, Abreu--Nigro, Hikita, Skandera]\label{thm:smooth}
For every smooth $w\in S_n$ and every $\lambda\vdash n$, $\phi^\lambda(\Cp w)\ge0$.
\end{proposition}

This is a direct combination of known results, which we record for later use. For codominant $w$ it is \cite[Prop.~25]{Ska25}, where Skandera deduces it from Hikita's theorem.

\begin{proof}
For smooth $w$, $\Cp w$ and $\Cp{w'}$ have the same character values for some codominant $w'$ (Abreu--Nigro \cite[Thm.~1.6]{AN}; see also Clearman--Hyatt--Shelton--Skandera \cite[Thm.~4.6]{CHSS}). By \cite[Cor.~7.2]{CHSS} there is a unit interval order $P$ with $X_P=\sum_\lambda\phi^\lambda(\Cp w)\,e_\lambda$, where $X_P$ is Stanley's chromatic symmetric function of the incomparability graph of $P$. These coefficients are non-negative by Hikita's proof of the Stanley--Stembridge conjecture \cite{Hik}.
\end{proof}

Proposition~\ref{thm:smooth}, and through it Theorem~\ref{thm:clean}, the reduction for the shapes $(3,2^b,1^c)$ in Section~\ref{sec:twocol} and the smooth case of Problem~\ref{prob:theta}, depend on Hikita's theorem \cite{Hik}. Theorem~\ref{thm:main} does not use it. For hooks, Hikita's theorem is not needed: $f^{(a)}_i(\Cp w)\ge0$ for every $w$, smooth or not, by Theorem~\ref{thm:main}. More generally, Skandera \cite[Thm.~5.5]{Ska08} showed that every immanant $d_f$ that is non-negative on $\TNN_n$ has non-negative coefficients at all smooth $w$. So Conjecture~\ref{conj:MN} for a shape $\lambda$ implies $c_{\lambda,w}\ge0$ at every smooth $w$; Proposition~\ref{thm:smooth} proves this unconditionally.

\subsection{Singular permutations: clean steps}\label{sec:clean}
Here $C'_w$ is the Kazhdan--Lusztig basis of the Hecke algebra $H_n(q)$, and $\ch(h)=\sum_\lambda\chi^\lambda_q(h)\,s_\lambda$ for $h\in H_n(q)$, where $\chi^\lambda_q$ are the irreducible characters of $H_n(q)$. At $q=1$, $\ch(C'_w)$ determines the vector $c_{\cdot,w}=(\phi^\lambda(\Cp w))_\lambda$ and conversely. Let $s=s_l=(l\ l{+}1)$, so that $ws$ exchanges the entries in positions $l,l+1$. Suppose $sw<w<ws$. The Kazhdan--Lusztig multiplication rule gives $C'_wC'_s=C'_{ws}+\sum_{z\in M(w,s)}\mu(z,w)C'_z$ and $C'_sC'_w=(q^{1/2}+q^{-1/2})C'_w$, where
\[
M(w,s)=\{z<w:\ zs<z,\ \mu(z,w)\ne0\}.
\]
Since characters are traces, Abreu and Nigro \cite[(3a)]{AN} obtain
\begin{equation}\label{eq:AN}
(q^{1/2}+q^{-1/2})\,\ch(C'_w)=\ch(C'_{ws})+\sum_{z\in M(w,s)}\mu(z,w)\,\ch(C'_z).
\end{equation}

\begin{definition}\label{def:clean}
The step $w\to ws$ is \emph{clean} if $M(w,s)=\emptyset$. Left steps $w\to sw$, with $ws<w<sw$, are defined in the same way.
\end{definition}

Each $\ch(C'_z)$ is a non-zero Schur-positive symmetric function \cite[Lemma~1.1]{Hai}, so at $q=1$ the step is clean iff $c_{\cdot,ws}=2c_{\cdot,w}$.

\begin{theorem}\label{thm:clean}
Let $y=y_0>y_1>\dots>y_k$ with each $y_{i-1}$ obtained from $y_i$ by a clean step, and $y_k$ smooth. Then $\ch(q^{\ell(y)/2}C'_y)=(1+q)^k\ch(q^{\ell(w')/2}C'_{w'})$ for a codominant $w'$, and $\phi^\lambda(\Cp y)=2^k\phi^\lambda(\Cp{w'})\ge0$ for every $\lambda\vdash n$.
\end{theorem}

\begin{proof}
For a clean step, \eqref{eq:AN} reads $\ch(q^{\ell(ws)/2}C'_{ws})=(1+q)\ch(q^{\ell(w)/2}C'_w)$. Induct on $k$; the case $k=0$ is \cite[Thm.~1.6]{AN} and Proposition~\ref{thm:smooth}.
\end{proof}

The case $k=1$ is \cite[Thm.~1.7]{AN}. Since a clean step gives $c_{\cdot,ws}=2c_{\cdot,w}$ also when $w$ is singular, Conjecture~\ref{conj:H} holds for $n$ iff $c_{\lambda,w}\ge0$ for every singular $w$ that is not the top of a clean step. This proves the second statement of Theorem~\ref{thm:smoothclean}. The rest of this section gives a local criterion for clean steps.

\begin{definition}\label{def:badcover}
Let $w(l)<w(l+1)$. A \emph{bad cover} of $w$ at $l$ is a permutation $z$ covered by $w$ in Bruhat order with $z(l)>z(l+1)$.
\end{definition}

Write $a=w(l)<b=w(l+1)$. A cover $z=w\,(i\ j)$ can be bad only if $i=l+1<j$, $w(j)<a$ and no $k\in(l+1,j)$ has $w(j)<w(k)<b$; or $i<l=j$, $w(i)>b$ and no $k\in(i,l)$ has $a<w(k)<w(i)$. Covers have $\mu=1$, so a clean step has no bad cover.

\begin{exa}\label{exa:clean}
List $c_{\cdot,w}$ in the order $\lambda=(4),(3,1),(2,2),(2,1,1),(1^4)$. For the smooth $w=3142$ and $l=2$, $sw=2143<w<ws=3412$, and $M(w,s)=\emptyset$: the step is clean, and indeed $c_{\cdot,3142}=(4,2,2,0,0)$ and $c_{\cdot,3412}=(8,4,4,0,0)$. For the smooth $w=2341$ and $l=1$, $sw=1342<w<ws=3241$, and $z=2143$ is a bad cover: it is covered by $w$ and $z(1)>z(2)$. Here $M(w,s)=\{2143\}$, $c_{\cdot,2341}=(4,2,2,0,0)$, $c_{\cdot,3241}=(8,4,0,0,0)$ and $c_{\cdot,2143}=(0,0,4,0,0)$, so \eqref{eq:AN} at $q=1$ reads $2(4,2,2,0,0)=(8,4,0,0,0)+(0,0,4,0,0)$ and the step is not clean. In line with \cite[Prop.~3.1]{AN}, the top $3412$ of the clean step is singular, and the top $3241$ of the other is smooth.
\end{exa}

\begin{proposition}\label{prop:4crown}
Let $z<w$ with $\ell(w)-\ell(z)=3$ and $\mu(z,w)\ne0$. Then $w$ and $z$ differ in exactly four positions $p_1<p_2<p_3<p_4$, and either $w$ restricted to them has pattern $3412$ and $z$ has pattern $1324$, or $w$ has pattern $4231$ and $z$ has pattern $2143$. The coatoms of $[z,w]$ are $w\,(p_1\ p_3)$, $w\,(p_2\ p_3)$, $w\,(p_2\ p_4)$, $w\,(p_1\ p_4)$ in the first case, and $w\,(p_3\ p_4)$, $w\,(p_2\ p_4)$, $w\,(p_1\ p_3)$, $w\,(p_1\ p_2)$ in the second. All of them are covered by $w$.
\end{proposition}

\begin{proof}
Since $P_{z,w}\ne1$, $X_w$ is singular at the point $z$. Every irreducible component of the singular locus of $X_w$ has codimension at least $3$, and the components of codimension exactly $3$ are the $X_z$ given by the two configurations above, with the regions they span empty \cite{BW03,Man01,KLR03,Cor03}. Because $\ell(w)-\ell(z)=3$, $X_z$ is such a component. The interval $[z,w]$ consists of the permutations that agree with $w$ outside the four positions and whose restriction lies in $[1324,3412]$, respectively $[2143,4231]$, of $S_4$, which gives the coatoms.
\end{proof}

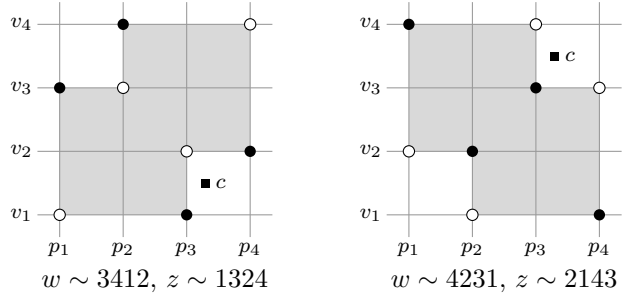
\begin{figure}[ht]
\centering
\begin{tikzpicture}[scale=0.42,
  blk/.style={circle,fill=black,inner sep=1.5pt},
  wht/.style={circle,draw=black,fill=white,inner sep=1.5pt}]
  \begin{scope}[xshift=0cm]
    \fill[black!15] (1,1) rectangle (3,3);
    \fill[black!15] (1,3) rectangle (3,5);
    \fill[black!15] (3,1) rectangle (5,3);
    \fill[black!15] (3,3) rectangle (5,5);
    \fill[black!15] (3,5) rectangle (5,7);
    \fill[black!15] (5,3) rectangle (7,5);
    \fill[black!15] (5,5) rectangle (7,7);
    \draw[black!40] (0.3,1) -- (7.7,1) (0.3,3) -- (7.7,3) (0.3,5) -- (7.7,5) (0.3,7) -- (7.7,7);
    \draw[black!40] (1,0.3) -- (1,7.7) (3,0.3) -- (3,7.7) (5,0.3) -- (5,7.7) (7,0.3) -- (7,7.7);
    \node[blk] at (1,5) {};
    \node[wht] at (1,1) {};
    \node[font=\scriptsize] at (1,-0.1) {$p_1$};
    \node[font=\scriptsize] at (-0.2,1) {$v_1$};
    \node[blk] at (3,7) {};
    \node[wht] at (3,5) {};
    \node[font=\scriptsize] at (3,-0.1) {$p_2$};
    \node[font=\scriptsize] at (-0.2,3) {$v_2$};
    \node[blk] at (5,1) {};
    \node[wht] at (5,3) {};
    \node[font=\scriptsize] at (5,-0.1) {$p_3$};
    \node[font=\scriptsize] at (-0.2,5) {$v_3$};
    \node[blk] at (7,3) {};
    \node[wht] at (7,7) {};
    \node[font=\scriptsize] at (7,-0.1) {$p_4$};
    \node[font=\scriptsize] at (-0.2,7) {$v_4$};
    \node[fill=black,rectangle,inner sep=1.6pt] at (5.6,2) {};
    \node[font=\scriptsize,right] at (5.6,2) {$c$};
    \node[font=\small] at (4,-1.1) {$w\sim3412$, $z\sim1324$};
  \end{scope}
  \begin{scope}[xshift=11cm]
    \fill[black!15] (1,3) rectangle (3,5);
    \fill[black!15] (1,5) rectangle (3,7);
    \fill[black!15] (3,1) rectangle (5,3);
    \fill[black!15] (3,3) rectangle (5,5);
    \fill[black!15] (3,5) rectangle (5,7);
    \fill[black!15] (5,1) rectangle (7,3);
    \fill[black!15] (5,3) rectangle (7,5);
    \draw[black!40] (0.3,1) -- (7.7,1) (0.3,3) -- (7.7,3) (0.3,5) -- (7.7,5) (0.3,7) -- (7.7,7);
    \draw[black!40] (1,0.3) -- (1,7.7) (3,0.3) -- (3,7.7) (5,0.3) -- (5,7.7) (7,0.3) -- (7,7.7);
    \node[blk] at (1,7) {};
    \node[wht] at (1,3) {};
    \node[font=\scriptsize] at (1,-0.1) {$p_1$};
    \node[font=\scriptsize] at (-0.2,1) {$v_1$};
    \node[blk] at (3,3) {};
    \node[wht] at (3,1) {};
    \node[font=\scriptsize] at (3,-0.1) {$p_2$};
    \node[font=\scriptsize] at (-0.2,3) {$v_2$};
    \node[blk] at (5,5) {};
    \node[wht] at (5,7) {};
    \node[font=\scriptsize] at (5,-0.1) {$p_3$};
    \node[font=\scriptsize] at (-0.2,5) {$v_3$};
    \node[blk] at (7,1) {};
    \node[wht] at (7,5) {};
    \node[font=\scriptsize] at (7,-0.1) {$p_4$};
    \node[font=\scriptsize] at (-0.2,7) {$v_4$};
    \node[fill=black,rectangle,inner sep=1.6pt] at (5.6,6) {};
    \node[font=\scriptsize,right] at (5.6,6) {$c$};
    \node[font=\small] at (4,-1.1) {$w\sim4231$, $z\sim2143$};
  \end{scope}
\end{tikzpicture}
\caption{The two configurations of Proposition~\ref{prop:4crown}. A point $(p,w(p))$ is drawn in column $p$ and row $w(p)$; black points are $w$ at the positions $p_1<\dots<p_4$, and hollow points are $z$. In every pair $z<w$ with $\ell(w)-\ell(z)=3$ and $\mu(z,w)\ne0$ for $n\le6$ ($1$, $18$ and $243$ pairs of each type for $n=4,5,6$), no other point of $w$ lies in a shaded cell; the white cells may contain points. The square marks a value $c$ at position $l+1=p_3+1$ in a white cell, the configuration of the first row of each block of the table in Lemma~\ref{lem:nolongmu}. On the left $v_1<c<v_2$, and $y=w\,(p_2\ p_3)$ is a bad cover with $y(l)=v_4>c$. On the right $v_3<c<v_4$, and $y=w\,(p_1\ p_3)$ is a bad cover with $y(l)=v_4>c$.}\label{fig:config}
\end{figure}

\begin{lemma}\label{lem:nolongmu}
Let $w(l)<w(l+1)$, and let $z<w$ with $\ell(w)-\ell(z)=3$, $\mu(z,w)\ne0$ and $z(l)>z(l+1)$. Then $w$ has a bad cover at $l$.
\end{lemma}

\begin{proof}
Use the notation of Proposition~\ref{prop:4crown} and Figure~\ref{fig:config}, let $v_1<v_2<v_3<v_4$ be the values in positions $D=\{p_1,\dots,p_4\}$, and let $c$ be the common value of $w$ and $z$ at the position of $\{l,l+1\}$ outside $D$, if there is one. Since the order of the entries in positions $l,l+1$ is reversed between $w$ and $z$, $\{l,l+1\}\cap D\ne\emptyset$. If $l,l+1\in D$, they are consecutive in $D$, but in neither configuration does a consecutive pair increase in $w$ and decrease in $z$. In the remaining cases, $y$ is a coatom from Proposition~\ref{prop:4crown}, hence covered by $w$, with $y(l)>y(l+1)$:
\begin{center}
\begin{tabular}{lll}
\toprule
configuration & position & bad cover \\
\midrule
$3412\to1324$ & $l=p_3$, $v_1<c<v_2$ & $w\,(p_2\ p_3)$, \ $y(l)=v_4$\\
 & $l=p_4$, $v_2<c<v_4$ & $w\,(p_2\ p_4)$, \ $y(l)=v_4$\\
 & $l+1=p_1$, $v_1<c<v_3$ & $w\,(p_1\ p_3)$, \ $y(l+1)=v_1$\\
 & $l+1=p_2$, $v_3<c<v_4$ & $w\,(p_2\ p_3)$, \ $y(l+1)=v_1$\\
$4231\to2143$ & $l=p_3$, $v_3<c<v_4$ & $w\,(p_1\ p_3)$, \ $y(l)=v_4$\\
 & $l=p_4$, $v_1<c<v_3$ & $w\,(p_3\ p_4)$, \ $y(l)=v_3$\\
 & $l+1=p_1$, $v_2<c<v_4$ & $w\,(p_1\ p_2)$, \ $y(l+1)=v_2$\\
 & $l+1=p_2$, $v_1<c<v_2$ & $w\,(p_2\ p_4)$, \ $y(l+1)=v_1$\\
\bottomrule
\end{tabular}
\end{center}
The other positions of $l$ or $l+1$ in $D$ are excluded because there $w(l)<w(l+1)$ and $z(l)>z(l+1)$ cannot both hold.
\end{proof}

\begin{proposition}\label{prop:cleancrit}
Let $sw<w<ws$ with $s=s_l$, and $\deg P_{e,w}\le1$. Then the step $w\to ws$ is clean iff $w$ has no bad cover at $l$.
\end{proposition}

\begin{proof}
A clean step has no bad cover. Conversely, by monotonicity of Kazhdan--Lusztig polynomials \cite[Cor.~3.7]{BM01}, $P_{e,w}-P_{z,w}$ has non-negative coefficients for $z\le w$, so $\deg P_{z,w}\le1$. Hence $\mu(z,w)\ne0$ forces $\ell(w)-\ell(z)\in\{1,3\}$. A $z\in M(w,s)$ at distance $1$ is a bad cover, and one at distance $3$ gives a bad cover by Lemma~\ref{lem:nolongmu}.
\end{proof}

For smooth $w$, Proposition~\ref{prop:cleancrit} is contained in \cite[Prop.~3.1]{AN}. The degree condition cannot simply be dropped: for $w=256413$ and $l=1$ there is no bad cover, but $z=215436$ has $\ell(w)-\ell(z)=5$ and $\mu(z,w)=1$, so the step is not clean. Starting from smooth permutations, the steps certified by Proposition~\ref{prop:cleancrit} reach $84\%$, $68\%$, $49\%$, $32\%$ and $19\%$ of the singular permutations for $n=5,\dots,9$ \cite{McN}. So the criterion handles a large but shrinking share of the singular permutations, and the singular $w$ that are not the top of a clean step are where Conjecture~\ref{conj:H} remains open.

\section{Two further families of shapes}\label{sec:twocol}

\begin{lemma}\label{lem:fangreen}
If $\lambda_1\le2$ and $u$ contains $321$, then $\phi^\lambda(\Cp u)=0$.
\end{lemma}

\begin{proof}
The map $\theta\colon\mathbb{C}[S_n]\to TL_n(2)$, $1+s_i\mapsto t_i$, is a surjective algebra homomorphism. It kills $\Cp u$ exactly when $u$ contains $321$, and the $\theta(\Cp w)$ for $321$-avoiding $w$ form a basis (Fan--Green \cite{FG}). So $\ker\theta=\operatorname{span}\{\Cp u:u\ni321\}$. A two-sided ideal of $\mathbb{C}[S_n]$ is a sum of isotypic components. The ideal $\ker\theta$ contains $\Cp{s_1s_2s_1}=\sum_{v\in S_3}v$, which acts non-trivially on $V^\mu$ iff the restriction of $\chi^\mu$ to $S_3=\langle s_1,s_2\rangle$ contains the trivial character. By Young's rule this multiplicity is $\langle\chi^\mu,\eta^{(3,1^{n-3})}\rangle=K_{\mu,(3,1^{n-3})}$, which is non-zero iff $\mu_1\ge3$. So $\ker\theta$ contains every component with $\mu_1\ge3$. Since $\dim TL_n(2)=\mathrm{Cat}_n=\sum_{\mu_1\le2}(f^\mu)^2$, comparing dimensions shows that $\ker\theta$ is exactly their sum. Finally, $\phi^\lambda$ involves only $\chi^\mu$ with $\mu\trianglelefteq\lambda$, hence with $\mu_1\le2$, so it vanishes on $\ker\theta$.
\end{proof}

In the ring of symmetric functions, $m_{(2^{b+1},1^c)}\,p_1=m_{(3,2^b,1^c)}+(b+2)\,m_{(2^{b+2},1^{c-1})}+(c+1)\,m_{(2^{b+1},1^{c+1})}$, where the middle term is absent when $c=0$. Hence
\begin{equation}\label{eq:onethree}
\Imm_{(3,2^b,1^c)}=\sum_jx_{jj}\Imm_{(2^{b+1},1^c)}(x[\hat\jmath,\hat\jmath])
-(b+2)\Imm_{(2^{b+2},1^{c-1})}-(c+1)\Imm_{(2^{b+1},1^{c+1})}.
\end{equation}

\begin{proposition}\label{prop:onethree}
For all $b,c\ge0$ and every $u\in S_{3+2b+c}$ that contains $321$, $\phi^{(3,2^b,1^c)}(\Cp u)\ge 0$.
\end{proposition}

\begin{proof}
At $u\ni321$, the two-column terms of \eqref{eq:onethree} vanish by Lemma~\ref{lem:fangreen}. The first term is KL-positive: $\Imm_{(2^{b+1},1^c)}$ is a non-negative combination of Temperley--Lieb immanants \cite{CSS}, which are Kazhdan--Lusztig immanants \cite{RS06}, so each $x_{jj}\Imm_{(2^{b+1},1^c)}(x[\hat\jmath,\hat\jmath])$ is KL-positive by Theorem~\ref{thm:R2}(1).
\end{proof}

So Conjecture~\ref{conj:H} for the shapes $(3,2^b,1^c)$ reduces to $321$-avoiding permutations. Those avoiding $312$ or $231$ are smooth and covered by Proposition~\ref{thm:smooth}. For $321$-avoiding, hexagon-avoiding $w$ with reduced word $s_{i_1}\cdots s_{i_l}$, Billey and Warrington \cite[Thm.~1]{BW} give $\Cp w=(1+s_{i_1})\cdots(1+s_{i_l})$, so $c_{\lambda,w}=\sum_{S\subseteq[l]}\phi^\lambda(\prod_{k\in S}s_{i_k})$.

\section{Beyond class functions}\label{sec:beyond}

The cones of Section~\ref{sec:cones} consist of class functions, and for these $\CKL=\CTNN$ when $n\le10$. Without the restriction to class functions this fails already for $n=4$ (Rhoades--Skandera \cite[\S7]{RS06}; Skandera and Zelevinsky give an explicit element, reported in \cite[(5.1)]{Ska08}). We recall the setting, give examples for $n=5$, and prove that the TNN-isolated permutations are exactly the smooth ones (Theorem~\ref{thm:isolated}). Let $\Pi_n$ be the span of the permutation monomials $x^v$. Every $F\in\Pi_n$ is uniquely $F=\sum_wa_w\Imm_w$, and we put
\[
\CKL(\Pi_n)=\{F:\ a_w\ge0\ \forall w\},\qquad \CTNN(\Pi_n)=\{F:\ F\ge0\text{ on }\TNN_n\}.
\]
Then $\CKL(\Pi_n)\subseteq\CTNN(\Pi_n)$, and the cones of Section~\ref{sec:cones} are their intersections with the immanants $d_\phi$, $\phi\in\CF_n$. Call $w$ \emph{TNN-isolated} if every $F\in\CTNN(\Pi_n)$ has $a_w\ge0$. So $\CKL(\Pi_n)=\CTNN(\Pi_n)$ iff every $w$ is TNN-isolated. Equivalently, since the evaluation vectors $(\Imm_v(A))_v$, $A\in\TNN_n$, generate the cone dual to $\CTNN(\Pi_n)$, $w$ is TNN-isolated iff the coordinate vector $e_w$ is a limit of positive combinations of evaluation vectors.

These cones go back to Stembridge \cite[\S5]{Ste92}, who studied the cone of total positivity in $\R S_n$, whose dual is $\CTNN(\Pi_n)$. Rhoades and Skandera \cite[\S7]{RS06} noted $\CKL(\Pi_n)\subsetneq\CTNN(\Pi_n)$ for $n\ge4$, observed that many $\Imm_w$ seem to span extreme rays of $\CTNN(\Pi_n)$, and asked whether $a_w\ge0$ for every $F\in\CTNN(\Pi_n)$ whenever $X_w$ is smooth \cite[Question~7.1]{RS06}. Skandera answered this question \cite{Ska08} and asked for the precise relation between the cones \cite[Problem~5.8]{Ska08}.

\begin{theorem}[Skandera {\cite[Thms.~5.4, 5.5]{Ska08}}; asked in {\cite[Question~7.1]{RS06}}]\label{thm:skandera}
For every $n$, every smooth $w\in S_n$ is TNN-isolated: there is $A_w\in\TNN_n$ with $\Imm_v(A_w)=\delta_{v,w}$ for all $v$.
\end{theorem}

Skandera's matrix $A_w$ is the path matrix of a planar network built from a factorisation of $w$. For codominant $w$, Proposition~\ref{prop:hessenberg} gives such a matrix directly, the antiadjacency matrix of a unit interval order.

A \emph{chart} is a factorisation into elementary factors $I+tE_{i+1,i}$ and $I+tE_{i,i+1}$ whose lower letters form a reduced word of $w_0$, as do its upper letters. By Fomin--Zelevinsky \cite{FZ}, positive parameters give every totally positive matrix up to a positive diagonal factor, which multiplies every element of $\Pi_n$ by its determinant. Since totally positive matrices are dense in $\TNN_n$, an $F\in\Pi_n$ whose expansion in the parameters of some chart has only non-negative coefficients is non-negative on $\TNN_n$.

\begin{theorem}\label{thm:beyond}
The elements
\begin{gather*}
\Imm_{15432}+\Imm_{43215}-\Imm_{45312},\qquad \Imm_{21543}+\Imm_{54321}-\Imm_{52431},\\ \Imm_{32154}+\Imm_{54321}-\Imm_{53241}
\end{gather*}
of $\Pi_5$ are non-negative on $\TNN_5$ and are not KL-positive. So $\CKL(\Pi_5)\subsetneq\CTNN(\Pi_5)$, and $45312$, $52431$, $53241$ are not TNN-isolated.
\end{theorem}

\begin{proof}
For each element there is a chart in which its expansion has only non-negative coefficients: the standard chart (lower word $s_4\,s_3s_4\,s_2s_3s_4\,s_1s_2s_3s_4$, followed by its reverse in the upper factors) for the first two, and for the third the chart
\[
L_1L_2U_1L_3U_3L_1U_2L_2U_4U_3L_1L_4L_3U_1L_2L_1U_2U_4U_1U_3,
\]
with $L_i(t)=I+tE_{i+1,i}$, $U_i(t)=I+tE_{i,i+1}$ and one parameter per factor; its lower letters $1,2,3,1,2,1,4,3,2,1$ and upper letters $1,3,2,4,3,1,2,4,1,3$ are reduced words of $w_0$. This is verified in exact integer arithmetic: the three expansions have $870$, $916$ and $10{,}611$ monomials, all with coefficients between $1$ and $16$ (\texttt{paper/hooks\_kl\_examples.py} in \cite{McNcode}). They are not KL-positive because the coefficients $-1$ are coefficients in the basis $\{\Imm_w\}$.
\end{proof}

All three non-isolated permutations are singular, as they must be by Theorem~\ref{thm:skandera}. We now show that Skandera's theorem is sharp for every $n$: no singular $w$ is TNN-isolated. The proof lifts two elements of $\CTNN(\Pi_4)$ along an occurrence of $3412$ or $4231$ in $w$, and controls the lift with Theorem~\ref{thm:R2}.

\begin{proposition}\label{prop:n4ineq}
On $\TNN_4$,
\[
\Imm_{3412}^2\le\Imm_{1432}\Imm_{3214},\qquad \Imm_{4231}^2\le3\,\Imm_{2143}\Imm_{4321},
\]
where $\Imm_{4321}=x_{14}x_{23}x_{32}x_{41}$. Hence $F=\tfrac12(\Imm_{1432}+\Imm_{3214})-\Imm_{3412}$ and $G_\varepsilon=\varepsilon\Imm_{2143}+\tfrac{3}{4\varepsilon}\Imm_{4321}-\Imm_{4231}$, for every $\varepsilon>0$, lie in $\CTNN(\Pi_4)$.
\end{proposition}

\begin{proof}
Use the chart (Figure~\ref{fig:network})
\[
L_3(t_1)L_2(t_2)L_1(t_3)L_3(t_4)L_2(t_5)L_3(t_6)\,U_3(t_7)U_2(t_8)U_1(t_9)U_3(t_{10})U_2(t_{11})U_3(t_{12}),
\]
where $L_i(t)=I+tE_{i+1,i}$ and $U_i(t)=I+tE_{i,i+1}$; the word $s_3s_2s_1s_3s_2s_3$ is a reduced word of $w_0$. In this chart $\Imm_{1432}$, $\Imm_{3214}$, $\Imm_{2143}$ and $\Imm_{4321}$ have only positive coefficients, $\Imm_{3412}=m$ is a single monomial, and $\Imm_{4231}=m_1+m_2+m_3$, where
\begin{gather*}
m=t_2t_3t_4t_5t_8t_9t_{10}t_{11},\qquad m_1=t_1t_2t_3^2t_5t_8t_9^2t_{11}t_{12},\\
m_2=t_1t_2t_3^2t_9^2t_{11}t_{12},\qquad m_3=t_1t_2t_3t_9t_{11}t_{12}.
\end{gather*}
The monomial $a=t_2t_4t_5t_8t_{10}t_{11}$ of $\Imm_{1432}$ and the monomial $b=t_2t_3^2t_4t_5t_8t_9^2t_{10}t_{11}$ of $\Imm_{3214}$ satisfy $m^2=ab$, so $\Imm_{3412}^2\le\Imm_{1432}\Imm_{3214}$ for positive parameters. The monomials $a_1=t_1t_3t_5^2t_8^2t_9t_{12}$ and $a_2=t_1t_3t_9t_{12}$ of $\Imm_{2143}$ and $b_1=t_1t_2^2t_3^3t_9^3t_{11}^2t_{12}$ and $b_2=t_1t_2^2t_3t_9t_{11}^2t_{12}$ of $\Imm_{4321}$ satisfy $m_1^2=a_1b_1$, $m_2^2=a_2b_1$ and $m_3^2=a_2b_2$. By the Cauchy--Schwarz inequality,
\[
\Imm_{4231}^2\le3(m_1^2+m_2^2+m_3^2)\le3(a_1+a_2)(b_1+b_2)\le3\Imm_{2143}\Imm_{4321} .
\]
A positive diagonal factor multiplies both sides of each inequality by $\det(D)^2$, so both hold on the totally positive matrices, and by density on $\TNN_4$. The last statement follows from $\sqrt{XY}\le\frac{X+Y}2$ and $\sqrt{3XY}\le\varepsilon X+\frac3{4\varepsilon}Y$ for $X,Y\ge0$, since Kazhdan--Lusztig immanants are non-negative on $\TNN_4$.
\end{proof}

\begin{figure}[ht]
\centering
\begin{tikzpicture}[xscale=0.78,yscale=0.95,>=stealth]
  \draw[thick] (0,-1) -- (13,-1);
  \node[left,font=\scriptsize] at (0,-1) {$1$};
  \node[right,font=\scriptsize] at (13,-1) {$1'$};
  \draw[thick] (0,-2) -- (13,-2);
  \node[left,font=\scriptsize] at (0,-2) {$2$};
  \node[right,font=\scriptsize] at (13,-2) {$2'$};
  \draw[thick] (0,-3) -- (13,-3);
  \node[left,font=\scriptsize] at (0,-3) {$3$};
  \node[right,font=\scriptsize] at (13,-3) {$3'$};
  \draw[thick] (0,-4) -- (13,-4);
  \node[left,font=\scriptsize] at (0,-4) {$4$};
  \node[right,font=\scriptsize] at (13,-4) {$4'$};
  \draw[->,thick] (0.7,-4) -- (1.3,-3);
  \node[font=\tiny,fill=white,inner sep=0.5pt] at (1,-3.5) {$t_{1}$};
  \draw[->,thick] (1.7,-3) -- (2.3,-2);
  \node[font=\tiny,fill=white,inner sep=0.5pt] at (2,-2.5) {$t_{2}$};
  \draw[->,thick] (2.7,-2) -- (3.3,-1);
  \node[font=\tiny,fill=white,inner sep=0.5pt] at (3,-1.5) {$t_{3}$};
  \draw[->,thick] (3.7,-4) -- (4.3,-3);
  \node[font=\tiny,fill=white,inner sep=0.5pt] at (4,-3.5) {$t_{4}$};
  \draw[->,thick] (4.7,-3) -- (5.3,-2);
  \node[font=\tiny,fill=white,inner sep=0.5pt] at (5,-2.5) {$t_{5}$};
  \draw[->,thick] (5.7,-4) -- (6.3,-3);
  \node[font=\tiny,fill=white,inner sep=0.5pt] at (6,-3.5) {$t_{6}$};
  \draw[->,thick] (6.7,-3) -- (7.3,-4);
  \node[font=\tiny,fill=white,inner sep=0.5pt] at (7,-3.5) {$t_{7}$};
  \draw[->,thick] (7.7,-2) -- (8.3,-3);
  \node[font=\tiny,fill=white,inner sep=0.5pt] at (8,-2.5) {$t_{8}$};
  \draw[->,thick] (8.7,-1) -- (9.3,-2);
  \node[font=\tiny,fill=white,inner sep=0.5pt] at (9,-1.5) {$t_{9}$};
  \draw[->,thick] (9.7,-3) -- (10.3,-4);
  \node[font=\tiny,fill=white,inner sep=0.5pt] at (10,-3.5) {$t_{10}$};
  \draw[->,thick] (10.7,-2) -- (11.3,-3);
  \node[font=\tiny,fill=white,inner sep=0.5pt] at (11,-2.5) {$t_{11}$};
  \draw[->,thick] (11.7,-3) -- (12.3,-4);
  \node[font=\tiny,fill=white,inner sep=0.5pt] at (12,-3.5) {$t_{12}$};
\end{tikzpicture}
\caption{The chart of Proposition~\ref{prop:n4ineq} as a planar network. Paths run from the sources $1,\dots,4$ on the left to the sinks $1',\dots,4'$ on the right; an edge labelled $t_k$ from line $a$ to line $b$ is the factor $I+t_kE_{a,b}$, and horizontal edges have weight $1$. The $(i,j)$ entry of the product of the twelve factors is the sum of the weights of the paths from $i$ to $j'$ (Lindstr\"om). The first six edges are the lower factors $L_i(t)=I+tE_{i+1,i}$, and the last six the upper factors $U_i(t)=I+tE_{i,i+1}$.}\label{fig:network}
\end{figure}
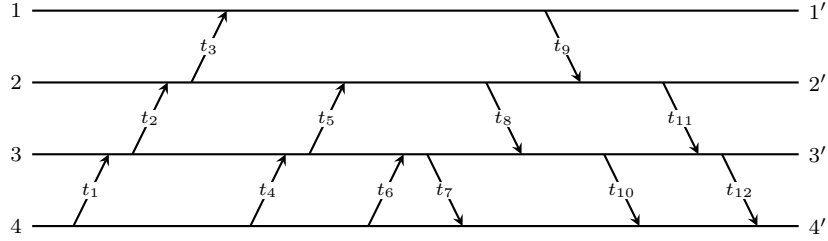

The Skandera--Zelevinsky element is $\Imm_{3214}+\Imm_{1432}-\Imm_{3412}=2F+\Imm_{3412}$, so $F$ is the stronger statement.

Now fix $w\in S_n$ and $I\subseteq[n]$ with $|I|=k$, and put $T=w(I)$. As in the proof of Theorem~\ref{thm:R2}, $\varphi(\sigma)=(\sigma_I,\sigma_{I^c})$ for $\sigma(I)=T$. Let $p=w_I$ and $w_c=w_{I^c}$, so that $w$ contains the pattern $p$ in the positions $I$. For $s\in S_k$ let $\sigma_s$ be the permutation with $\sigma_s(I)=T$ and $\varphi(\sigma_s)=(s,w_c)$: it agrees with $w$ outside $I$ and has pattern $s$ in the positions $I$. For $u\in S_k$ let $A(u)$ be the coefficient of $\Imm_w$ in $\Imm_u(x[I,T])\,\Imm_{w_c}(x[I^c,T^c])$. By Theorem~\ref{thm:R2}, $A(u)\in\N$.

\begin{lemma}\label{lem:lift}
\begin{enumerate}
\item $A(p)=1$, and $A(u)=0$ for every $u>p$.
\item For every $u\in S_k$, $A(u)+\sum_{u'>u}A(u')\,P_{u,u'}(1)\le P_{\sigma_u,w}(1)$.
\end{enumerate}
\end{lemma}

\begin{proof}
By the proof of Theorem~\ref{thm:R2}, $A(u)=a_y$ for the $y$ with $\varphi(y)=(u,w_c)$. Since the Kazhdan--Lusztig polynomials of $W'$ are products, \eqref{eq:BB} at $\sigma=\sigma_u$ reads
\[
P_{\sigma_u,w}(1)=\sum_ya_y\,P_{u,\varphi(y)_1}(1)\,P_{w_c,\varphi(y)_2}(1),
\]
a sum of non-negative terms. Keeping only the $y$ with $\varphi(y)_2=w_c$ gives~(2).

For~(1), the map $\varphi^{-1}$ is order-preserving from $S_k\times S_{n-k}$ to the coset $\{\sigma:\sigma(I)=T\}$. Indeed, if $s(a)<s(b)$ for $a<b$, then multiplying $s$ on the right by the transposition $(a\ b)$ corresponds to multiplying $\sigma$ on the right by the transposition of the $a$-th and $b$-th elements of $I$, whose entries in $\sigma$ also increase; the same holds in $I^c$. By the definition of Kazhdan--Lusztig immanants, the product $\Imm_p(x[I,T])\Imm_{w_c}(x[I^c,T^c])=\sum_\sigma f(\sigma)x^\sigma$ has $f(\sigma)\ne0$ only if $\sigma_I\ge p$ and $\sigma_{I^c}\ge w_c$, hence only if $\sigma\ge w$, and $f(w)=1$. Write the product as $\sum_yc_y\Imm_y$. Since $\Imm_y$ involves only monomials $x^v$ with $v\ge y$, a Bruhat-minimal $y$ with $c_y\ne0$ has $c_y=f(y)$, so $y\ge w$. Hence $c_y\ne0$ implies $y\ge w$, and comparing the coefficients of $x^w$ gives $A(p)=c_w=f(w)=1$. Finally,~(2) for $u=p$, where $\sigma_p=w$, reads $1+\sum_{u'>p}A(u')P_{p,u'}(1)\le1$. Since $P_{p,u'}(1)\ge1$, $A(u')=0$ for $u'>p$.
\end{proof}

\begin{lemma}\label{lem:smoothswap}
Let $i<j$ with $w(i)>w(j)$, and let $R$ be the set of $k\in(i,j)$ with $w(j)<w(k)<w(i)$. Then the Schubert variety $X_w$ is smooth at $w\,(i\ j)$ iff $w(k)>w(l)$ for all $k<l$ in $R$, that is, iff $w$ has no occurrence of $4231$ in positions $i<k<l<j$. In that case $P_{w(i\,j),w}=1$.
\end{lemma}

\begin{proof}
Put $x=w\,(i\ j)$ and $r_y(p,v)=\#\{p'\le p:y(p')\ge v\}$. Then $y\le w$ iff $r_y\le r_w$ pointwise \cite[Thm.~2.1.5]{BjBr}. If $p<q$ and $y(p)<y(q)$, then $r_{y(p\,q)}=r_y+\mathbf 1_{[p,q-1]\times[y(p)+1,y(q)]}$. Hence $r_x=r_w-\mathbf 1_B$ with $B=[i,j-1]\times[w(j)+1,w(i)]$, and for $p<q$ with $x(p)<x(q)$,
\[
x\,(p\ q)\le w\iff[p,q-1]\times[x(p)+1,x(q)]\subseteq B,
\]
that is, iff the points $(p,x(p))$ and $(q,x(q))$ both lie in the box $[i,j]\times[w(j),w(i)]$. The points of $x$ in this box are $(i,w(j))$, $(j,w(i))$ and $(k,w(k))$, $k\in R$. So the number of such pairs is $(|R|+1)+|R|+\operatorname{asc}(R)$, where $\operatorname{asc}(R)$ is the number of pairs $k<l$ in $R$ with $w(k)<w(l)$. All $\ell(x)$ transpositions $(p\ q)$ with $x(p)>x(q)$ give elements below $x\le w$. By Lakshmibai and Seshadri \cite{LSe}, the tangent space of $X_w$ at $x$ has dimension $\#\{(p,q):p<q,\ x\,(p\ q)\le w\}=\ell(x)+2|R|+1+\operatorname{asc}(R)$, while $\dim X_w=\ell(w)=\ell(x)+2|R|+1$. So $X_w$ is smooth at $x$ iff $\operatorname{asc}(R)=0$, and then $P_{x,w}=1$ \cite{KL80}. A pair $k<l$ in $R$ with $w(k)<w(l)$ is exactly an occurrence $(i,k,l,j)$ of $4231$.
\end{proof}

\begin{figure}[ht]
\centering
\begin{tikzpicture}[scale=0.55,
  blk/.style={circle,fill=black,inner sep=1.5pt},
  wht/.style={circle,draw=black,fill=white,inner sep=1.5pt}]
  \begin{scope}[xshift=0cm]
    \fill[black!12] (1.5,1.5) rectangle (5.5,5.5);
    \draw[black!35] (0.5,0.5) grid[step=1] (5.5,5.5);
    \node[blk] at (1,1) {};
    \node[blk] at (2,5) {};
    \node[blk] at (3,4) {};
    \node[blk] at (4,3) {};
    \node[blk] at (5,2) {};
    \node[wht] at (2,2) {};
    \node[wht] at (5,5) {};
    \node[font=\tiny] at (1,0.05) {1};
    \node[font=\tiny] at (0.05,1) {1};
    \node[font=\tiny] at (2,0.05) {2};
    \node[font=\tiny] at (0.05,2) {2};
    \node[font=\tiny] at (3,0.05) {3};
    \node[font=\tiny] at (0.05,3) {3};
    \node[font=\tiny] at (4,0.05) {4};
    \node[font=\tiny] at (0.05,4) {4};
    \node[font=\tiny] at (5,0.05) {5};
    \node[font=\tiny] at (0.05,5) {5};
    \node[font=\small] at (3.0,-0.7) {$w=15432$};
  \end{scope}
  \begin{scope}[xshift=8cm]
    \fill[black!12] (1.5,1.5) rectangle (5.5,5.5);
    \draw[black!35] (0.5,0.5) grid[step=1] (5.5,5.5);
    \node[blk] at (1,1) {};
    \node[blk] at (2,5) {};
    \node[blk] at (3,3) {};
    \node[blk] at (4,4) {};
    \node[blk] at (5,2) {};
    \node[wht] at (2,2) {};
    \node[wht] at (5,5) {};
    \node[font=\tiny] at (1,0.05) {1};
    \node[font=\tiny] at (0.05,1) {1};
    \node[font=\tiny] at (2,0.05) {2};
    \node[font=\tiny] at (0.05,2) {2};
    \node[font=\tiny] at (3,0.05) {3};
    \node[font=\tiny] at (0.05,3) {3};
    \node[font=\tiny] at (4,0.05) {4};
    \node[font=\tiny] at (0.05,4) {4};
    \node[font=\tiny] at (5,0.05) {5};
    \node[font=\tiny] at (0.05,5) {5};
    \node[font=\small] at (3.0,-0.7) {$w=15342$};
  \end{scope}
\end{tikzpicture}
\caption{Lemma~\ref{lem:smoothswap} for $i=2$, $j=5$. The point $(p,w(p))$ is drawn in column $p$ and row $w(p)$; black points are $w$, and the hollow points are the two entries of $x=w\,(2\ 5)$ that differ from $w$. The shaded box is $[i,j]\times[w(j),w(i)]$, and $R=\{3,4\}$. Left: $R$ is decreasing, so $\operatorname{asc}(R)=0$, and the tangent space at $x=12435$ has dimension $\ell(x)+2|R|+1=6=\ell(w)$: $X_w$ is smooth at $x$. Right: the ascent $(3,4)$ in $R$ is the occurrence $(2,3,4,5)$ of $4231$, the tangent space at $x=e$ has dimension $0+5+1=6>\ell(w)=5$, and $P_{e,w}=1+q$.}\label{fig:smoothswap}
\end{figure}
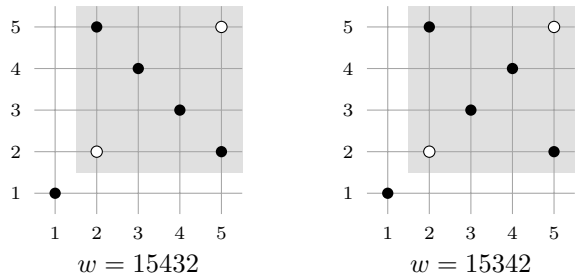

We include a proof of Lemma~\ref{lem:smoothswap} for completeness. It is a direct consequence of the tangent space formula; Figure~\ref{fig:smoothswap} shows both cases.

\begin{theorem}\label{thm:isolated}
For every $n$, a permutation $w\in S_n$ is TNN-isolated iff the Schubert variety $X_w$ is smooth, that is, iff $w$ avoids $3412$ and $4231$ \cite{LS}. Equivalently, for every singular $w$ there is an element of $\CTNN(\Pi_n)$ whose coefficient at $\Imm_w$ is negative.
\end{theorem}

\begin{proof}
Smooth permutations are TNN-isolated by Theorem~\ref{thm:skandera}. Let $w$ be singular. For $W=\sum_uc_u\Imm_u\in\CTNN(\Pi_4)$ and an occurrence of a pattern of length $4$ in positions $I$, put $H=W(x[I,T])\,\Imm_{w_c}(x[I^c,T^c])$. Submatrices of totally non-negative matrices are totally non-negative, and $\Imm_{w_c}\ge0$ on $\TNN_{n-4}$, so $H\in\CTNN(\Pi_n)$. Its coefficient at $\Imm_w$ is $\sum_uc_uA(u)$.

If $w$ contains $4231$, take an occurrence of $4231$ and $W=G_\varepsilon$. By Lemma~\ref{lem:lift}(1), with $4321>4231$, the coefficient is $\varepsilon A(2143)+\tfrac3{4\varepsilon}A(4321)-A(4231)=\varepsilon A(2143)-1$, which is negative for small $\varepsilon>0$.

Otherwise $w$ contains $3412$ and avoids $4231$. Take any occurrence $i_1<i_2<i_3<i_4$ of $3412$ and $W=F$. Here $\sigma_{1432}=w\,(i_1\ i_3)$ and $\sigma_{3214}=w\,(i_2\ i_4)$. Since $w$ avoids $4231$, Lemma~\ref{lem:smoothswap} gives $P_{\sigma_{1432},w}=P_{\sigma_{3214},w}=1$. Lemma~\ref{lem:lift}(2) for $u=1432<3412$, together with $A(3412)=1$ and $P_{1432,3412}=1$, gives $A(1432)+1\le1$, so $A(1432)=0$, and likewise $A(3214)=0$. The coefficient is $\tfrac12A(1432)+\tfrac12A(3214)-A(3412)=-1$.
\end{proof}

\begin{exa}\label{exa:lift}
Let $w=45312$, one of the permutations of Theorem~\ref{thm:beyond}. It avoids $4231$ and contains $3412$ in the positions $I=\{1,2,4,5\}$, with $T=w(I)=\{1,2,4,5\}$, $p=3412$ and $w_c=e\in S_1$ (Figure~\ref{fig:lift}). Then $\sigma_{1432}=w\,(1\ 4)=15342$ and $\sigma_{3214}=w\,(2\ 5)=42315$, and $P_{\sigma_{1432},w}=P_{\sigma_{3214},w}=1$, as Lemma~\ref{lem:smoothswap} predicts. Computing the products directly, the coefficient $A(u)$ of $\Imm_w$ in $\Imm_u(x[I,T])\,x_{33}$ is $1$ for $u=3412$ and $0$ for every other $u\in S_4$. So
\[
H=\bigl(\tfrac12\Imm_{1432}+\tfrac12\Imm_{3214}-\Imm_{3412}\bigr)(x[I,T])\;x_{33}
\]
lies in $\CTNN(\Pi_5)$, and its coefficient at $\Imm_{45312}$ is $-1$; it is the only negative coefficient of $H$. For the $4231$ case, $w=52431$ contains $4231$ in the positions $\{1,2,4,5\}$, and there $A(4231)=1$ and $A(u)=0$ for $u\ne4231$, so the lift of $G_\varepsilon$ has coefficient $-1$ at $\Imm_{52431}$ for every $\varepsilon>0$.
\end{exa}

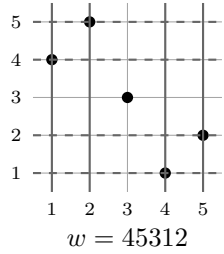
\begin{figure}[ht]
\centering
\begin{tikzpicture}[scale=0.5,
  blk/.style={circle,fill=black,inner sep=1.5pt},
  wht/.style={circle,draw=black,fill=white,inner sep=1.5pt}]
  \begin{scope}[xshift=0cm]
    \draw[black!35] (0.5,0.5) grid[step=1] (5.5,5.5);
    \node[blk] at (1,4) {};
    \node[blk] at (2,5) {};
    \node[blk] at (3,3) {};
    \node[blk] at (4,1) {};
    \node[blk] at (5,2) {};
    \node[font=\tiny] at (1,0.05) {1};
    \node[font=\tiny] at (0.05,1) {1};
    \node[font=\tiny] at (2,0.05) {2};
    \node[font=\tiny] at (0.05,2) {2};
    \node[font=\tiny] at (3,0.05) {3};
    \node[font=\tiny] at (0.05,3) {3};
    \node[font=\tiny] at (4,0.05) {4};
    \node[font=\tiny] at (0.05,4) {4};
    \node[font=\tiny] at (5,0.05) {5};
    \node[font=\tiny] at (0.05,5) {5};
    \node[font=\small] at (3.0,-0.7) {$w=45312$};
    \foreach \p in {1,2,4,5}{\draw[thick,black!60] (\p,0.5)--(\p,5.5);}
    \foreach \v in {1,2,4,5}{\draw[thick,black!60,dashed] (0.5,\v)--(5.5,\v);}
  \end{scope}
\end{tikzpicture}
\caption{The occurrence of $3412$ in $w=45312$ used in Example~\ref{exa:lift}: positions $I=\{1,2,4,5\}$ (solid lines) and values $T=w(I)=\{1,2,4,5\}$ (dashed lines). The remaining point $(3,3)$ gives $w_c=e\in S_1$.}\label{fig:lift}
\end{figure}

So the full-span cones differ in the coordinate of every singular $w$, while for class functions $\CKL=\CTNN$ when $n\le10$ \cite{McN}. The companion paper \cite{McN} proves the cases $n=4,5$ of Theorem~\ref{thm:isolated} by a different method, a vertex criterion on Newton polytopes. It also checks the proof by exact computation. For every singular $w$ with $n\le10$, the lifted element ($2F$ for $3412$, and $\Imm_{2143}+\Imm_{4321}-\Imm_{4231}=G_1+\frac14\Imm_{4321}$ for $4231$) has a negative coefficient at $\Imm_w$ already at the first occurrence tried. Lemma~\ref{lem:smoothswap} is checked for every inversion of every $w$ with $n\le10$. For $4231$ the proof works at every occurrence. For $3412$ the elements $\varepsilon\Imm_{1432}+\frac1{4\varepsilon}\Imm_{3214}-\Imm_{3412}$, which lie in $\CTNN(\Pi_4)$ by Proposition~\ref{prop:n4ineq}, show that it suffices that one of $w\,(i_1\ i_3)$ and $w\,(i_2\ i_4)$ is a smooth point of $X_w$.

\section{Conjectures and open problems}\label{sec:open}

\subsection{Explicit expansions for hooks}
Theorem~\ref{thm:main} does not make the Kazhdan--Lusztig coefficients of $f^{(a)}_i$ explicit, except for $a=2$. Computation suggests that for $a\ge2$, $f^{(a)}_i(\Cp w)=0$ unless $w$ moves a single block of consecutive letters containing $i$ (observed for $n\le8$). For $a=3$, \eqref{eq:cycle} gives $f^{(3)}_i(\Cp w)=\sum_{j\ne i}B_{ij}(w)-f^{(2)}_i(\Cp w)$, where $B_{ij}(w)\ge0$ is the coefficient of $\Imm_w$ in $x_{ij}x_{ji}\det x[\{i,j\}^c,\{i,j\}^c]$.

\begin{conjecture}\label{conj:B}
Let $i<j$.
\begin{enumerate}
\item If $j-i\ge2$, let $Q=\{i-1,i,j,j+1\}\cap[n]$. Then $B_{ij}(w)=1$ if $w$ fixes $[n]\setminus Q$ and its pattern on $Q$ lies in the Bruhat interval $[(i\;j),(\min Q\;\max Q)]$ of $S_{|Q|}$, and $B_{ij}(w)=0$ otherwise. For interior pairs this interval is $[1324,4231]$, with $16$ elements.
\item If $j=i+1$, $B_{ij}$ equals $1$ on the patterns $\ge(i\;i{+}1)$ on $\{i-1,\dots,i+2\}$, except the longest element when $|Q|=4$, and $2$ on the pattern $3412$. It is also $1$ on the $5$-cycles
  \begin{gather*}
  (i{-}2,i,i{+}2,i{-}1,i{+}1),\ (i{-}2,i{+}1,i{-}1,i{+}2,i)\qquad (i\ge3),\\
  (i{-}1,i{+}1,i{+}3,i,i{+}2),\ (i{-}1,i{+}2,i,i{+}3,i{+}1)\qquad (i+3\le n),
  \end{gather*}
  and $0$ elsewhere.
\end{enumerate}
\end{conjecture}

Table~\ref{tab:Badj} lists the support of $B_{i,i+1}$ in the interior case $3\le i\le n-3$, where all three kinds of term occur.

\begin{table}[ht]
\centering\small
\begin{tabular}{@{}lll@{}}
\toprule
$w$ & $B_{i,i+1}(w)$ & number\\
\midrule
fixes $[n]\setminus Q$, pattern on $Q$ in $[1324,4321)$, not $3412$ & $1$ & $18$\\[2pt]
\multicolumn{3}{@{}l@{}}{\quad $1324,1342,1423,1432,2314,2341,2413,2431,3124,$}\\
\multicolumn{3}{@{}l@{}}{\quad $3142,3214,3241,3421,4123,4132,4213,4231,4312$}\\[2pt]
fixes $[n]\setminus Q$, pattern $3412$ on $Q$ & $2$ & $1$\\
one of the four $5$-cycles of Conjecture~\ref{conj:B}(2) & $1$ & $4$\\
\midrule
any other $w$ & $0$ & \\
\bottomrule
\end{tabular}
\caption{The conjectured support of $B_{i,i+1}$ for $3\le i\le n-3$, with $Q=\{i-1,i,i+1,i+2\}$. For $n=6$, $i=3$ it has $23$ elements, and the four $5$-cycles are $345126$, $451236$, $145623$, $156234$.}\label{tab:Badj}
\end{table}

This matches the computation for every pair $(i,j)$ with $n\le8$. Two parts are proved: $B_{ij}(w)=0$ if $w$ has a left or right descent $s_k$ with $\{k,k+1\}\cap\{i,j\}=\emptyset$; and $B_{ij}$ is supported on minimal double coset representatives for the Young subgroup with blocks $[1,i-1],\{i\},[i+1,j-1],\{j\},[j+1,n]$. With Lemma~\ref{lem:a2}, Conjecture~\ref{conj:B} gives a closed form for $\Imm_{(3,1^{n-3})}$ in the Kazhdan--Lusztig basis. It would follow from an inertness property: for a cycle $\gamma$ on $T\subset[n]$, a point $k\notin T$ at distance at least $2$ from $T$ appears to be a fixed point of every $w$ in the support of $x_\gamma\det x[T^c,T^c]$, with deletion of $k$ giving the corresponding coefficients in $S_{n-1}$, if $\min T<k<\max T$, or no two elements of $T$ are adjacent, or $\gamma$ is an adjacent transposition and $k$ is at distance at least $3$ (checked for $2$- and $3$-cycles with $n\le8$).

\subsection{A rooted refinement}
Cut the cycles of $\sigma\in S_n$ into \emph{arcs}, runs $v,\sigma(v),\dots,\sigma^{l-1}(v)$ of consecutive elements of a cycle. Expanding each $p_k$ in the $h$-basis by brick tabloids gives $\phi^\lambda(\sigma)=(-1)^{n-\ell(\lambda)}\sgn(\sigma)N_\lambda(\sigma)$, where $N_\lambda(\sigma)$ is the number of cuttings into arcs of lengths $\lambda$. For a vertex $v$, an integer $a\ge1$ and $\mu\vdash n-a$, let
$G^{(a;\mu)}_v(\sigma)=(-1)^{n-1-\ell(\mu)}\sgn(\sigma)M$, where $M$ is the number of cuttings with an arc of length $a$ starting at $v$ and the other arcs of lengths $\mu$, the arcs of equal length $\ge2$ being ordered. Sorting the cuttings by the arc through $v$ gives, for every $v$,
\[
\phi^\lambda=\sum_{a}\frac{a}{\operatorname{aut}(\lambda\setminus a)}\,G^{(a;\lambda\setminus a)}_v,
\]
the sum over the distinct parts $a$ of $\lambda$, with $\operatorname{aut}(\mu)$ the product of the factorials of the multiplicities of the parts $\ge2$ of $\mu$. For hooks and $a=\lambda_1$ these are the per-vertex functions of Theorem~\ref{thm:main}.

\begin{conjecture}\label{conj:rooted}
For $v\in\{1,n\}$, every $a\ge1$, every $\mu\vdash n-a$ and every $w\in S_n$, $G^{(a;\mu)}_v(\Cp w)\ge0$.
\end{conjecture}

Conjecture~\ref{conj:rooted} implies Conjecture~\ref{conj:H}. It is a combinatorial form, at $q=1$, of a conjecture of Abreu and Nigro. For every $w$, the decomposition theorem for the map $Y_w(X)\to\mathbb P^{n-1}$, $V_\bullet\mapsto V_1$, from a regular semisimple Lusztig variety splits $\ch IH^*(Y_w(X))=\ch(q^{\ell(w)/2}C'_w)$ as $\sum_i[n-i]_q h_{n-i}\ch(\rho_{i,w})$ for representations $\rho_{i,w}$ of $S_i$ \cite[Cor.~1.9]{AN22b}, and they conjecture that the local systems come from permutation representations \cite[Conj.~1.4]{AN22b}. Our decomposition has the same shape, $\ch(\Cp w)=\sum_a a\,h_a\sum_\mu G^{(a;\mu)}_1(\Cp w)\operatorname{aut}(\mu)^{-1}h_\mu$. For codominant $w$ the pieces agree with Abreu--Nigro's $g$-functions for every $n$. Let $m$ be the Hessenberg function of $w$, so that $\Cp w$ is the sum of the $v$ with $v(i)\le m(i)$ for all $i$ (Proposition~\ref{prop:hessenberg}). At $q=1$, and after applying $\omega$, \cite[Thm.~1.2]{AN23} expresses the character of their piece with support of codimension $n-a$ as $\sum_{v\le w}[L(v)\ge a]\,(-1)^{L(v)-a}e_{L(v)-a}\,p_{\nu(v)}$, where $L(v)$ is the length of the cycle of $v$ through $1$ and $\nu(v)$ is the cycle type of the other cycles. Cutting a path with $k$ vertices into arcs in all ways gives $\sum_{\alpha\models k}(-1)^{k-\ell(\alpha)}h_\alpha=e_k$, and cutting a $k$-cycle gives $\sum_{\alpha\models k}(-1)^{\ell(\alpha)-1}\alpha_1h_\alpha=p_k$, so this sum equals $\sum_\mu G^{(a;\mu)}_1(\Cp w)\operatorname{aut}(\mu)^{-1}h_\mu$. (We also checked the agreement in graded form for $n\le5$.) The $g$-functions are the images under $\omega$ of these characters, and they are $e$-positive at $q=1$ by Huh, Hwang, Kim, Kim and Oh \cite{HHKKO}. This is the case $v=1$. For $v=n$, conjugation by $w_0$ exchanges the vertices $1$ and $n$, and inverting $\sigma$ reverses its arcs, which does not change the number of cuttings with an arc of length $a$ at a given vertex; since $P_{v,w}=P_{v^{-1},w^{-1}}=P_{w_0vw_0,\,w_0ww_0}$, this gives $G^{(a;\mu)}_n(\Cp w)=G^{(a;\mu)}_1(\Cp{w_0w^{-1}w_0})$. The permutation $w_0w^{-1}w_0$ avoids $312$ if $w$ does. So Conjecture~\ref{conj:rooted} holds for codominant $w$. We checked it for every $w$ with $n\le8$. The restriction to $v\in\{1,n\}$ is necessary: at the interior vertices $1<v<n$ the corresponding inequality fails, already for $n=3$, where $G^{(1;(2))}_2(\Cp w)=-1$ for $w=231$ and for $w=312$. The pieces $G^{(a;\mu)}_1(\Cp w)/\operatorname{aut}(\mu)$ are integers for $n\le6$, and for $n\le6$ they lift to palindromic graded pieces in $\N[q]$ compatible with $\ch(q^{\ell(w)/2}C'_w)$, as the decomposition theorem predicts. Two structural facts hold for every $n$. If $sw<w<ws$ and the step is clean, then $G^{(a;\mu)}_v(\Cp{ws})=G^{(a;\mu)}_v(\Cp w)+G^{(a;\mu)}_{s(v)}(\Cp w)$, because $\Cp w\,s=s\,\Cp w\,s$ and conjugation by $s$ relabels the vertices. For $s\ne s_1$, which fixes vertex $1$, the functions $G^{(a;\mu)}_1$ satisfy the analogue at $q=1$ of the relation \eqref{eq:AN}. For $n\le6$ these relations determine them from their values at codominant $w$.

\subsection{Problems}
Conjectures~\ref{conj:B} and~\ref{conj:rooted} are concrete proposals supported by computation. Of the problems below, the first four are the main directions; the others are narrower or more exploratory targets.
\begin{enumerate}
\item\label{prob:KT} Find a structural reason why, for class functions, KL-positivity and TNN-positivity should agree ($\CKL=\CTNN$), independently of Conjecture~\ref{conj:H}. This would make Conjectures~\ref{conj:MN} and~\ref{conj:H} equivalent. Averaging does not obviously help: $\TNN_n$ is not closed under conjugation by permutation matrices, and the evident symmetries preserving both $\TNN_n$ and the Kazhdan--Lusztig immanants, $x\mapsto x^T$ and $x\mapsto w_0xw_0$ (acting by $w\mapsto w^{-1}$ and $w\mapsto w_0ww_0$), form a group of order $4$, and averaging over it does not produce class functions.
\item Prove Conjecture~\ref{conj:rooted}, for example by a restricted modular law for the vertex-$1$ functions in the spirit of \cite{HHKKO}, or by identifying the permutation representations of \cite[Conj.~1.4]{AN22b}. The hook-style induction of Section~\ref{sec:hooks} does not extend: for non-hook $\lambda$ the coset pieces of $G^{(a;\mu)}_v$ are not KL-positive.
\item\label{prob:q} Prove Haiman's $q$-version for hooks, $\phi^{(a,1^{n-a})}_q(q^{\ell(w)/2}C'_w(q))\in\N[q]$. The proof of Theorem~\ref{thm:main} works at $q=1$ only. These polynomials are automatically symmetric about $q^{\ell(w)/2}$, because $C'_w$ is invariant under the bar involution; the full conjecture \cite[Conj.~2.1]{Hai} also asks that they be unimodal. Haiman verified both for $n\le7$ \cite{Hai}, and the companion paper for $n\le10$ \cite{McN}.
\item\label{prob:clean} Clean steps.
  \begin{enumerate}
  \item Describe the singular $w$ that are not the top of a clean step, and prove $c_{\lambda,w}\ge0$ for them; by Theorem~\ref{thm:clean} this is equivalent to Conjecture~\ref{conj:H}.
  \item Replace the condition $\deg P_{e,w}\le1$ in Proposition~\ref{prop:cleancrit} by a pattern condition.
  \end{enumerate}
\item Prove $\phi^{(3,2^b,1^c)}(\Cp w)\ge0$ for $321$-avoiding $w$ containing both $312$ and $231$; by Proposition~\ref{prop:onethree} this is Conjecture~\ref{conj:H} for these shapes.
\item\label{prob:theta} Prove that $\theta_\ell=\sum_{\ell(\mu)=\ell}\phi^\mu$ is KL-positive for every $n$ (Skandera's Problem~30 \cite{Ska25}, at $q=1$). It holds at smooth $w$ by Proposition~\ref{thm:smooth}, and for $n\le10$, since $\theta_\ell$ is a sum of monomial characters, by \cite{McN}. Since $\sum_\ell t^\ell\theta_\ell(\sigma)=\prod_{C}(1-(1-t)^{|C|})$ over the cycles $C$ of $\sigma$, $\theta_\ell=\sum_{|T|=n-\ell}g_T$ with $g_T(\sigma)=(-1)^{|T|}\sgn(\sigma)[\text{no cycle of }\sigma\text{ lies in }T]$, and every $g_T$ is KL-positive for $n\le8$.
\item Describe the cone $\CTNN(\Pi_n)$, for example its extreme rays (Stembridge \cite[Question~5.4]{Ste92}; Skandera \cite[Problem~5.8]{Ska08}). By Theorem~\ref{thm:isolated}, $\CKL(\Pi_n)$ and $\CTNN(\Pi_n)$ differ in the coordinate of every singular $w$. A first step is to find the extreme rays for $n=4,5$ and look for families that persist for all $n$.
\end{enumerate}

\section*{Declaration of generative AI and AI-assisted technologies in the manuscript preparation process}
During the preparation of this work the author used Claude (Anthropic), through Claude Code, to write and test the software, to run and cross-check the computations, to search and summarise the literature, to check arguments against computed examples, and to draft and edit the text. After using this tool, the author reviewed and edited the content as needed and takes full responsibility for the content of the published article.

\end{document}